\documentclass[12pt]{amsart}

\title{Commutative control data for smoothly locally trivial stratified spaces}
\author{Yoav Zimhony}

\usepackage{mathtools}
\usepackage{nicefrac}
\usepackage{enumitem}
\usepackage{biblatex}
\usepackage[hidelinks]{hyperref}
\usepackage[T1]{fontenc}

\calclayout

\newcommand{\R}{\mathbb{R}}
\newcommand{\N}{\mathbb{N}}
\newcommand{\E}{\mathcal{E}}
\newcommand{\g}{\mathfrak{g}}
\newcommand{\h}{\mathfrak{h}}
\renewcommand{\S}{\mathcal{S}}
\renewcommand{\L}{\mathcal{L}}
\newcommand\norm[1]{\lVert#1\rVert}

\newcommand{\subscript}[2]{${#1} _ #2$}

\newtheorem{thm}{Theorem}[section]
\newtheorem{lem}[thm]{Lemma}

\newtheorem*{claim_nn}{Claim}
\newtheorem{cor}[thm]{Corollary}
\newtheorem{prop}[thm]{Proposition}

\theoremstyle{remark}
\newtheorem{example}[thm]{Example}
\newtheorem{rem}[thm]{Remark}

\newtheorem{notation}[thm]{Notation}

\theoremstyle{definition}
\newtheorem{definition}[thm]{Definition}
\newtheorem*{question}{Question}

\theoremstyle{plain} % just in case the style had changed
\newcommand{\thistheoremname}{}
\newtheorem*{genericthm}{\thistheoremname}
\newenvironment{namedthm}[1]
  {\renewcommand{\thistheoremname}{#1}%
   \begin{genericthm}}
  {\end{genericthm}}

\begin{document}

\maketitle

\begin{abstract}
For a compact Lie group $G$ and a Hamiltonian $G$-space $M$ with momentum map $\mu:M\to \g^*$, we prove that the zero level set $\mu^{-1}(0)$ and the critical set $\text{Crit}\norm{\mu}^2$ of the norm squared momentum map are neighbourhood smooth weak deformation retracts. 
To this end we show that these subsets, stratified by orbit types, satisfy a condition stronger than Whitney (B) regularity --- \textit{smooth local triviality with conical fibers}.
Using this condition we construct control data in the sense of Mather with the additional properties that the fiber-wise multiplications by scalars, coming from the tubular neighbourhood structures, preserve strata and commute with each other. We use this control data to obtain the neighbourhood smooth weak deformation retraction.
Finally, such structures for the zero level set $\mu^{-1}(0)$ reduce to similar structures for the reduced space $\mu^{-1}(0)/G$, yielding a similar result for the reduced space and its stratified subspaces.
\end{abstract}

\section{Introduction}
\subsection{Motivation and background}
Let $G$ be a compact Lie group and $M$ a Hamiltonian $G$-space with momentum map $\mu:M\to \g^*$.
Sjamaar and Lerman \cite[Theorem~2.1]{sjamaar1991stratified} proved that the orbit type stratification of $M$ induces Whitney (B) stratifications of the zero level $\mu^{-1}(0)$ and of the reduced space $\mu^{-1}(0)/G$.

Given an $\text{Ad}$-invariant inner product on $\g$ and the induced ${\text{Ad}}^*$-invariant inner product on $\g^*$, denote by $\text{Crit}\norm{\mu}^2$ the critical set of the norm square of the momentum map. The orbit type stratification of $M$ also induces a $G$-invariant Whitney (B) stratification of $\text{Crit}\norm{\mu}^2$, and the zero level set $\mu^{-1}(0)$ is a union of connected components of $\text{Crit}\norm{\mu}^2$. See Theorem \ref{crit_set_stratification_thm}.

The motivating question of this paper is the following:
\begin{question}
For each of $\mu^{-1}(0)$ and $\text{Crit}\norm{\mu}^2$, is there a $G$-invariant neighbourhood and a $G$-equivariant smooth homotopy inverse to the inclusion into this neighbourhood?
\end{question}
When these sets are submanifolds, a positive answer to this question follows from the tubular neighbourhood theorem. A positive answer for $\text{Crit}\norm{\mu}^2$ was conjectured by Harada and Karshon \cite[Remark~4.23]{harada_karshon}, and it implies that their localization theorem for Duistermaat--Heckmann distributions \cite[Theorem~4.24]{harada_karshon} applies to $Z=\text{Crit}\norm{\mu}^2$ without additional assumptions. See similar results by Paradan \cite{PARADAN2000401} and Woodward \cite{woodward_gradient}.

Topologically, if we omit the requirement that the homotopy inverse be smooth and instead only require continuity, then this was proved by Woodward \cite[Appendix~A]{woodward_gradient} and Lerman \cite{lerman_gradient}, who attribute the proof to Duistermaat. They show that the gradient flow of the norm squared momentum map defines a continuous deformation retraction from a neighbourhood onto $\text{Crit}\norm{\mu}^2$.

One can obtain continuous deformation retractions also using the stratified structures on these spaces. For a general Whitney (B) stratified subset, Goresky \cite[Proposition~0.5.1]{goresky_thesis} used Mather's control data \cite{mather2012notes} to construct a structure called \textit{a family of lines}, which allowed him to obtain a continuous weak deformation retraction from a neighbourhood onto that subset. 
Pflaum and Wilkin \cite{pflaum2017equivariant} used similar ideas to construct a $G$-equivariant continuous strong deformation retraction when the stratification is $G$-invariant. Both methods yield maps that are smooth when restricted to each stratum.

Smoothly, we recall that a subset of a manifold is a smooth neighbourhood retract if and only if it is a submanifold, so a smooth strong deformation retraction does not exist for general stratified subsets. Moreover, we suspect that smooth weak deformation retractions do not exist for general Whitney (B) stratified subsets; the fast spiral \cite[Example~1.1.13]{pflaum_book} might be an example.

\subsection{Results and methods}
We prove that the answer for our motivating question is yes. This is carried out in the following steps.

\begin{namedthm}{Theorems \ref{crit_set_stratification_thm} and \ref{zero_level_set_stratification_thm}}
The stratifications of  $\mu^{-1}(0)$ and $\text{Crit}\norm{\mu}^2$ by orbit types are smoothly locally trivial with conical fibers.
\end{namedthm}
\textit{Smooth local triviality with conical fibers} is a local condition on stratified subsets, stronger than Whitney~(B) regularity.
It reads that for every stratum $X$ and every $p\in X$, a neighbourhood of $p$ in $M$ is diffeomorphic to $\R^k \times \R^{n-k}$ such that:
\begin{itemize}
    \item $X$ is identified with $\R^k\times\{0\}$.
    \item All nearby strata $Y$ (i.e., $X\subset \overline{Y}$) are of the form $\R^k\times C_Y$ for a subset $C_Y\subset \R^{n-k}\setminus \{0\}$ invariant under multiplication by $t\in(0,1)$.
\end{itemize}

We use this local property to obtain a global structure:
\begin{namedthm}{Theorems \ref{tangential_control_data_thm} and \ref{commutative_control_data_theorem}}
A subset with a $G$-invariant stratification that is smoothly locally trivial with conical fibers admits a $G$-equivariant commutative tangential control data.
\end{namedthm}
A \textit{$G$-equivariant commutative tangential control data} consists of the following data. For each stratum $X$ we have:
\begin{itemize}
    \item A $G$-invariant tubular neighbourhood $T_X\subset M$.
    \item A $G$-equivariant fiber-wise multiplication by scalars $m_X:T_X\times\R_{\geq 0}\to T_X$.
    \item A $G$-invariant fiber-wise distance-squared function $\rho_X:T_X\to\R_{\geq 0}$.
\end{itemize}
such that the following compatibility conditions are satisfied. For two strata $X,Y$ with $X\subset \overline{Y}$, in the domains of definition:
\begin{itemize}
    \item $m_X^0\circ m_Y^t = m_X^0$ for $t\geq 0$.
    \item $\rho_X\circ m_Y^t = \rho_X$ for $t\geq 0$.
    \item $\rho_Y\circ m_X^s = \rho_Y$ for $s > 0$.
    \item \textit{Tangential}: the map $m_X^s:T_X\to T_X$ preservers $Y\cap T_X$.
    \item \textit{Commutative}: the maps $m_X^s, m_Y^t$ commute for $t \geq 0, s > 0$.
\end{itemize}
See Subsection \ref{control_data_subsection} for full definitions. We note that Mather's control data is required to satisfy
only the first two bullets, and only with $t=0$.

The final step of our answer to the motivating question is:
\begin{namedthm}{Theorem \ref{weak_deformation_theorem}}
A subset with a $G$-invariant stratification and a $G$-equivariant commutative tangential control data is a $G$-equivariant neighbourhood smooth weak deformation retract.
\end{namedthm}
As a corollary, given a closed subset which is a union of strata, its inclusion map into the stratified subset has a smooth homotopy inverse (Remark \ref{stratified_subspace_deformation}).

The last section aims to extend the theory to reduced spaces. For zero level sets $\mu^{-1}(0)$, our control data reduces to similar data for reduced spaces $\mu^{-1}(0)/G$.
The obtained structure consists of maps on $\mu^{-1}(0)/G$, smooth in a suitable sense, which satisfy properties similar to those of the commutative tangential control data (Subsection \ref{reduced_space_control_data_subsection}). We also prove that the stratification of the reduced space with its smooth structure is \textit{smoothly locally trivial with quasi-homogenuous fibers} (Theorem \ref{reduced_space_local_structure_thm}).

A key technical tool used in this paper is the notion of Euler-like vector fields along submanifolds, introduced by Bursztyn, Lima and Meinrenken \cite{bursztyn2019splitting}. These vector fields have the following important properties:
\begin{itemize}
    \item An Euler-like vector field along $N$ induces a tubular neighbourhood of $N$ such that the flow of the vector field becomes the fiber-wise multiplication by $\exp(t)$.
    \item Being Euler-like is a local property and is preserved under convex combinations, so one can patch local Euler-like vector fields to a global one using partitions of unity.
\end{itemize}
Put together, these properties give a convenient way to patch local tubular neighbourhoods to a global one, while retaining information regarding the fiber-wise multiplication by scalars. We believe this method is simpler and more explicit than the one used by Mather \cite[Proposition~6.1]{mather2012notes} for similar purposes.

\subsection{Outline}
In Section \ref{euler_like_section} we recall the definition of Euler-like vector fields and develop some technical tools related to them. 
In Section \ref{stratified_section} we present the setting and the structures we aim to construct in the following sections. 
In Section \ref{tangential_control_data_section} we construct tangential control data for a stratified space $C$ which is smoothly locally trivial with conical fibers. In Section \ref{commutative_control_data_section} we modify this control data so that it is also commutative. In Section \ref{weak_deformation_section} we use the modified control data to construct a neighbourhood smooth weak deformation retraction to $C$.
In Section \ref{orbit_type_section} we prove that the stratifications of $\mu^{-1}(0)$ and $\text{Crit}\norm{\mu}^2$ are smoothly locally trivial with conical fibers.
Finally, in Section \ref{extension_section} we extend the theory so that it applies to reduced spaces $\mu^{-1}(0)/G$.

\subsection{Acknowledgements} 
I would like to thank my advisor, Yael Karshon, for introducing me to this problem and for many illuminating and helpful discussions. This research is partially funded by the United States -- Israel Binational Science Foundation.

\section{Tubular neighbourhoods and Euler-like vector fields}\label{euler_like_section}
We use the following conventions and notation. For a topological space $M$ and a subset $A\subset M$ we denote its closure by $\overline{A}$. For a vector field $X$ on a manifold $M$, we denote its flow by $\Psi_X^t:M\to M$.

We recall some definitions and results by Bursztyn, Lima and Meinrenken \cite{bursztyn2019splitting, MEINRENKEN2021224}. 
Let $M$ be a manifold of dimension $m$ and let $N\subset M$ be a submanifold of dimension $n$ which is closed in $M$. Let $\nu_N=\nu(M,N)$ be the normal bundle of $N$ in $M$ and identify $N\subset \nu_N$ with the zero section. Let 
\[
m^t:\nu_N\to \nu_N,\quad t\geq 0
\]
denote the multiplication by $t$ in $\nu_N$. 

For another manifold $M'$, a submanifold $N'\subset M'$ and a map $f:M\to M'$ taking $N$ to $N'$, let $\nu(f):\nu(M,N)\to\nu(M',N')$ be the \textit{linear approximation} of $f$.
\begin{definition}
A \textbf{tubular neighbourhood} of $N\subset M$ consists of a star-shaped open neighbourhood $U\subset \nu_N$ of the zero section (i.e., a neighbourhood invariant under $m^t$ for $t\in [0,1]$) and an embedding $\Phi:U\to M$ taking the zero section $N\subset \nu_N$ to $N\subset M$ such that the induced map
\[
\nu(\Phi):\nu_N\simeq \nu(U,N) \to \nu(M,N)=\nu_N
\]
is the identity.

A tubular neighbourhood is \textbf{complete} if $U=\nu_N$.
\end{definition}
\begin{definition}[{\cite[Example~2.1]{bursztyn2019splitting}}]
The \textbf{Euler vector field $\E$} on $\nu_N$ is the vector field with flow $\Psi_\E^t=m^{\exp{t}}$. In local bundle coordinates, with $x^i$ the coordinates on the fiber and $y^j$ those in the base direction, it is given by
\[
\E = \sum_{i}x^i\frac{\partial}{\partial x^i}.
\]
\end{definition}

Let $X$ be a vector field on $M$ tangent to $N$. In \cite{bursztyn2019splitting} it is shown that the normal bundle functor $\nu$ can be used to obtain a vector field $\nu(X)$ on $\nu_N$ with the property that the local flow of $\nu(X)$ is the linear approximation of the local flow of $X$.
\begin{definition}[{\cite[Definition~3.1]{MEINRENKEN2021224}}]\label{euler_like}
An \textbf{Euler-like vector field along $N$} is a vector field $X$ on $M$ tangent to $N$ such that $\nu(X)=\E$. 
\end{definition}
\begin{rem}\label{euler_like_deriv}
$X$ being Euler-like is equivalent to the following property --- for all $f\in C^\infty(M)$ vanishing on $N$, the function
\[
f - \L_Xf
\]
vanishes to order $2$ on $N$.
\end{rem}
\begin{rem}
We do not require an Euler-like vector field $X$ to be complete (cf. \cite[Definition~2.6]{bursztyn2019splitting}).
\end{rem}
\begin{prop}\label{euler_like_linearity}\hfill
\begin{enumerate}
    \item Let $X_1,\dots,X_r$ be Euler-like vector fields along $N$ and let $\varphi_1,\dots,\varphi_r$ be a partition of unity. 
    Then $\sum_{i=1}^r \varphi_i\cdot X_i$ is also Euler-like along $N$. 
    \item Let $X$ be an Euler-like vector field and let $G$ be compact Lie group. 
    Then the vector field obtained by averaging $X$ over $G$ is also Euler-like along $N$.
\end{enumerate}
\end{prop}
\begin{proof}
These follow by the linearity of the definition of Euler-like along $N$.
\end{proof}
\begin{thm}[\cite{bursztyn2019splitting,MEINRENKEN2021224}]\label{euler_like_tubular}
Let $X$ be an Euler-like vector field along $N$. Then there exists a unique maximal tubular neighbourhood $(U,\Phi)$ such that
\[
\E\sim_\Phi X.
\]

Moreover, the open set $\Phi(U)\subset M$ is given by
\[
\{p\in M : \lim_{t\to-\infty}\Psi_X^t(p)\text{ exists and is in }N\}.
\]

If $X$ is complete, the tubular neighbourhood
is complete. If $N$ and $X$ are $G$-invariant, the tubular neighbourhood embedding is $G$-equivariant.
\end{thm}

\subsection{Convenient Euler-like vector fields}
\begin{definition}
We call a smooth $\rho:\nu_N\to\R_{\geq 0}$ a \textbf{distance function} if:
\begin{enumerate}
    \item $\rho^{-1}(0)=N$
    \item $\rho\circ m^t = t^2 \cdot \rho$.
\end{enumerate}

A \textbf{tubular neighbourhood with distance} is a tubular neighbourhood with a distance function on $\nu_N$.
\end{definition}
\begin{rem}
A more intuitive name would be distance-square function but we use distance function for simplicity.
\end{rem}
\begin{rem}
Using Hadamard's lemma at $0$ and homogeneity one can show that each such distance function is of the form
\[
\rho(x) = g(x,x)
\]
for some fiber-wise inner product $g$ on $\nu_N$. 
\end{rem}
\begin{prop}\label{distance_derivative}
Let $\rho$ be a distance function. Then
\[
\L_\E\rho = 2\rho
\]
\end{prop}
\begin{proof}
This is Euler's homogeneous function theorem.
\end{proof}
\begin{prop}\label{distance_proper}
Let $\rho:\nu_N\to\R_{\geq 0}$ be a distance function. Then 
\[
(\rho, m^0):\nu_N\to \R_{\geq 0}\times N
\]
is proper.
\end{prop}
\begin{proof}
Using the projections to $\R_{\geq 0}$ and to $N$ we can reduce to the following case --- let $A\subset N$ be a compact subspace and let $b\in\R_{\geq 0}$. Then
\[
\rho^{-1}([0,b])\cap \nu_N\lvert_A
\]
is compact.

For any $q\in A$ let $U_q\subset N$ be neighbourhood of $q$ such that $\nu_N\lvert_{\overline{U_q}}$ is trivial. Since $A$ is compact there exists a finite cover of $A$ by such $U_{q_i}$ and it is enough to show that each
\[
\rho^{-1}([0,b]) \cap \nu_N\lvert_{\overline{U_{q_i}}\cap A}
\]
is compact. So we may assume $A\subset \overline{U_q}$. 

Identify $\nu_N\lvert_{\overline{U_q}}$ with $\overline{U_q}\times \R^{m-n}$. The set $A \times S^{m-n-1}$ is compact so $\rho$ has a minimum $c>0$ on it. Using the homogeneity of $\rho$ we get that 
\[
\rho^{-1}([0,c]) \cap \nu_N\lvert_A \subset A \times \overline{B^{m-n}_1}
\]
and therefore
\[
\rho^{-1}([0,b]) \cap \nu_N\lvert_A \subset A \times \overline{B^{m-n}_{b/c}}
\]
so $\rho^{-1}([0,b]) \cap \nu_N\lvert_A$ is compact as a closed subset of a compact set.
\end{proof}
\begin{definition}
Let $\Phi:\nu_N\to M$ be a complete tubular neighbourhood embedding. It is \textbf{extendable} in $M$ if there exist an open neighbourhood $U$ of $\overline{\Phi(\nu_N)}$ and a smooth extension of 
\[
m^0\circ \Phi^{-1}:\Phi(\nu_N)\to N
\]
to $U$.
\end{definition}
\begin{example}
Let $M=\R^2\setminus \{(0,0)\}$ and $N=\R_{>0}\times\{0\}$. Then $\R_{>0}\times\R\subset \R^2$ is the image of a complete tubular neighbourhood of $N$ which is not extendable --- the boundary contains the $Y$-axis without 0.
\end{example}
\begin{prop}\label{distance_close}
Let $\Phi:\nu_N\to M$ be a complete tubular neighbourhood that is extendable and let $\rho$ be a distance function.

Then for all $b\in\R_{\geq 0}$ the set $\Phi\left(\rho^{-1}([0,b])\right)$ is closed in $M$.
\end{prop}
\begin{proof}
Let ${\{v_i\}}_{i\in\N} \subset \Phi\left(\rho^{-1}([0,b])\right)$ be a sequence which converges to $v_0$. We show that $v_0\in \Phi\left(\rho^{-1}([0,b])\right)$.

Let $\pi:U\to N$ be the extension of $m^0\circ \Phi^{-1}$ to some neighbourhood of $\overline{\Phi(\nu_N)}$. Then $\pi(v_0)\in N$ is well defined and ${\{\pi(v_i)\}}_{i\in\N}$ converges to it. 

Take $W$ to be a neighbourhood of $\pi(v_0)$ in $N$ with a compact closure. By Proposition \ref{distance_proper} the set $\rho^{-1}([0,b]) \cap \nu_N\lvert_{\overline{W}}$ is compact so ${\{\Phi^{-1}(v_i)\}}_{i\in\N}$ has a sub-sequence which converges to a point in it. But the limit in $M$ is unique so $v_0\in \Phi\left(\rho^{-1}([0,b])\right)$.
\end{proof}
\begin{cor}\label{extending_functions}
Let $(\Phi, \rho)$ be an extendable tubular neighbourhood with distance and let $f:\R_{\geq 0} \to \R$ be a smooth function such that 
\[
\exists b\in\R_{\geq 0},c\in\R:\: f([b,\infty))=\{c\}.
\]

Then $f\circ\rho\circ\Phi^{-1}:\Phi(\nu_N)\to \R$ can be smoothly extended to $M$ by setting it to $c$ outside of $\Phi\left(\rho^{-1}([0,b])\right)$.
\end{cor}
\begin{proof}
The extension is smooth on both sets in the cover 
\[
\{M\setminus\Phi\left(\rho^{-1}([0,b])\right), \Phi(\nu_N)\}.
\]
which is open by Proposition \ref{distance_close}.
\end{proof}
\begin{definition}
Let $X$ be a complete Euler-like vector field along $N$ and let $\Phi:\nu_N\to M$ be the complete tubular neighbourhood induced by $X$. We call $X$ \textbf{convenient} if the following conditions hold:
\begin{enumerate}[label=(\subscript{CE}{{\arabic*}})]
    \item\label{conv_euler_extendable} The tubular neighbourhood $\Phi:\nu_N\to M$ is extendable
    \item\label{conv_euler_conical} The following sets coincide
\[
\{X\neq 0\} = \{p\in M\setminus N: \lim_{t\to -\infty}\Psi_X^t(p)\text{ exists and is in }N\}
\]
\end{enumerate}
\end{definition}
\begin{rem}
Completeness, condition \ref{conv_euler_conical} and Theorem \ref{euler_like_tubular} imply
\[
\Phi(\nu_N) = \{X\neq 0\}\sqcup N
\]
\end{rem}
\begin{definition}
Let $X$ be a convenient Euler-like vector field along $N$ and $\Phi:\nu_N\to M$ its induced tubular neighbourhood embedding. We call the data $\left(X, \Phi \right)$ a \textbf{convenient tubular neighbourhood} of $N$ in $M$.
\end{definition}
\begin{notation}\label{bump_function}
Let $0<a<b$. We denote by $h_{a,b}:\R_{\geq 0}\to [0,1]$ a smooth monotone bump function that is $\equiv 1$ on $[0,a]$ and $\equiv 0$ on $[b,\infty)$.
\end{notation}

\begin{lem}\label{convenient_euler}
Let $V$ be a neighbourhood of $N$ and let $X$ be an Euler-like vector field along $N$.

Then there exists a smooth function $\varphi:M\to [0,1]$ that satisfies:
\begin{enumerate}
    \item $\varphi$ is supported in $V$.
    \item $\varphi \equiv 1$ on a neighbourhood of $N$.
    \item The vector field $\varphi\cdot X$, defined on $M$, is a convenient Euler-like vector field along $N$.
\end{enumerate}

If $G$ is a compact Lie group acting on $M$, the sets $N,V$ are $G$-invariant and the vector field $X$ is $G$-invariant then $\varphi$ can be chosen to be $G$-invariant. 
\end{lem}
\begin{proof}
Using Theorem \ref{euler_like_tubular}, let $U\subset \nu_N$ and $\Phi:U\to M$ be the tubular neighbourhood induced by $X$. By normality of $M$ and the assumption that $N$ is closed in $M$ there exist a neighbourhood $W$ of $N$ in $M$ such that $\overline{W}\subset V\cap \Phi(U)$.

Let $\rho$ be any distance function on $\nu_N$ and let $\delta:N\to \R_{>0}$ be a smooth function such that
\[
\{p\in\nu_N : \rho(p) < \delta(m^0(p))\} \subset \Phi^{-1}(W)
\]
i.e., for each $q\in N$, the set $\{\rho < \delta(q) \}$ in the fiber $\nu_{N,q}$ is contained in $\Phi^{-1}(W)$. 

Now let $h \coloneqq h_{\nicefrac{1}{3}, \nicefrac{2}{3}}$ and define $\varphi':\nu_N\to [0,1]$ by
\[
\varphi'(x) = h\left( \frac{\rho(x)}{\delta(m^0(x))}\right).
\]

We claim that $\varphi'\circ \Phi^{-1}:\Phi(U)\to[0,1]$ can be smoothly extended to $M$ by setting it to $0$ outside of $\Phi(U)$. Indeed, the above extension is smooth on $M\setminus\overline{W}$ since it is constant there, it is smooth on $\Phi(U)$ by definition and these open sets form an open cover of $M$. Denote the above extension by $\varphi:M\to [0,1]$, we now show it satisfies the requirements.

\underline{$\varphi$ is supported in $V$}: This follows from construction.

\underline{$\varphi\cdot X$ is Euler-like}: Being Euler-like is a local condition near $N$ and $\varphi\cdot X = X$ in a small neighbourhood of each $q\in N$.

\underline{The vector field $\varphi\cdot X$ is complete}: Let $p\in M$. If $p\notin \text{supp}\varphi$ then $\left(\varphi\cdot X\right)(p)\equiv0$ in a neighbourhood of $p$ and the local flow around $p$ is complete.
Otherwise $p\in W\subset \Phi(U)$ so we can define $q=m^0(\Phi^{-1}(p))$. 

Since $\E\sim_\Phi X$ and $\E$ is tangent to the fibers of $\nu_N$ over $N$ it is enough to show that $\left(\varphi'\cdot \E\right)\lvert_{\nu_{N,q} \cap U}$ is complete. Note that
\[
\text{supp}\left(\varphi'\lvert_{\nu_{N,q}\cap U}\right) \subset \rho^{-1}([0,\delta(q)]) \cap \nu_{N,q}
\]
here the latter set is compact by Proposition \ref{distance_proper} and the former is a closed subset of it and thus compact. We get that $\left(\varphi'\cdot \E\right)\lvert_{\nu_{N,q}}$ has compact support and is therefore complete.

\underline{Condition \ref{conv_euler_extendable} is satisfied}: By construction, the tubular neighbourhood $\widetilde{\Phi}$ induced by $\varphi\cdot X$ is contained in $W$ so its closure is contained in $\Phi(U)$. The integral curves of $\varphi\cdot X$ are contained in those of $X$ and so the two projections
\[
m^0\circ \Phi^{-1},\: m^0\circ \widetilde{\Phi}^{-1}
\]
coincide on $\widetilde{\Phi}(\nu_N)$. This is exactly an extension of $m^0\circ \widetilde{\Phi}^{-1}$ to a neighbourhood of $\overline{\widetilde{\Phi}(\nu_N)}$.

\underline{Condition \ref{conv_euler_conical} is satisfied}: For the inclusion $\supset$, if $p\in M\setminus N$ is a point with flow limit in $N$ then necessarily $X(p)\neq 0$ otherwise it stays in place. 

For the other direction let $p\in M$ with $X(p)\neq 0$. By construction $p\in\Phi(U)$ so we may identify $p$ with its preimage in $U$ and $X$ with $\varphi'\cdot\E$. We show that the limit of $p$ under the flow of $\varphi'\cdot\E$ lies in $N$. 

Denote $\delta_0 = \frac{\delta(m^0(p))}{3}$. Since $X=\E$ on $\rho^{-1}([0,\delta_0])$ and the limit of the flow of $\E$ lies in $N$ for all points it is enough to show that $(\rho\circ \Psi_X^{-t_0})(p) < \delta_0$ for some $t_0>0$. 

Assume by way of contradiction that $\left(\rho\circ \Psi_X^{-t}\right)(p) > \delta_0$ for all $t>0$. One can check that $\frac{d}{dt}\left(\varphi'\circ \Psi_X^{-t}\right)(p) \geq 0$ and therefore $\left(\varphi'\circ \Psi_X^{-t}\right)(p)$ is non-decreasing in $t$. Using the above and Proposition \ref{distance_derivative} we get
\begin{align*}
\frac{d}{dt} \left(\rho\circ \Psi_X^{-t}\right)(p) & = -\L_X\rho\left( \Psi_X^{-t}(p) \right) \\
&= -\left(\varphi'\circ \Psi_X^{-t}\right)(p) \L_\E\rho\left( \Psi_X^{-t}(p) \right) \\
&= -2\left(\varphi'\circ \Psi_X^{-t}\right)(p)\rho\left( \Psi_X^{-t}(p)\right) \\
& \leq -2\varphi'(p)\left(\rho\circ \Psi_X^{-t}\right)(p) \\
& \leq -2\varphi'(p)\delta_0 <0
\end{align*}
so it follows that 
\[
\forall t\in \R: \left(\rho\circ \Psi_X^{-t}\right)(p) \leq \rho(p) -2\varphi'(p)\delta_0t
\]
which is a contradiction.

In the $G$-invariant case one only needs to take $W,\rho,\delta$ to be $G$-invariant. For $W$ this can be achieved by intersecting it with all its $G$-translates and for $\rho,\delta$ by averaging over $G$.
\end{proof}

\subsection{Shrinking of tubular neighbourhoods}\hfill

Let $\left(X, \Phi, \rho \right)$ be a convenient Euler-like vector field, its induced tubular neighbourhood and a distance function on $\nu_N$. Let $0<a<b$ and write $h\coloneqq h_{a,b}$.

By Corollary \ref{extending_functions} the function $(h\circ\rho\circ \Phi^{-1})$ can be smoothly extended to $M$ by setting it to $0$ outside of $\Phi\left(\rho^{-1}([0,b])\right)$. Similarly to the proof of Lemma \ref{convenient_euler}, denoting this extension by $\varphi$ we get that $\widetilde{X}\coloneqq\varphi\cdot X$ is also a convenient Euler-like vector field along $N$. Let $\widetilde{\Phi}$ be its induced tubular neighbourhood.
\begin{definition}
The data $\left(\widetilde{X}, \widetilde{\Phi}, \rho \right)$ is the \textbf{shrinking of $\left(X, \Phi, \rho \right)$ by $h$}.
\end{definition}
\begin{rem}
\hfill
\begin{enumerate}
    \item $\widetilde{\Phi}(\nu_N) \subset \Phi(\nu_N)$ as a fiber sub-bundle using the projection $m^0\circ \Phi^{-1}$.
    \item $\rho:\nu_N\to \R_{\geq 0}$ is the same function on the normal bundle but the functions
\[
\rho\circ\Phi^{-1},\:\rho\circ\widetilde{\Phi}^{-1}
\]
are different on $\widetilde{\Phi}(\nu_N)\setminus \widetilde{\Phi}\left(\rho^{-1}([0,a])\right)$.
\end{enumerate}
\end{rem}
\begin{lem}\label{shrinking_distance_level_sets}
Let $\left(\widetilde{X}, \widetilde{\Phi}, \rho \right)$ be the shrinking of $\left(X, \Phi, \rho \right)$ by $h$.

Then each level set of $\rho\circ\widetilde{\Phi}^{-1} $ inside $\widetilde{\Phi}(\nu_N)$ is also a level set of $\rho\circ\Phi^{-1}$.
\end{lem}
\begin{proof}
We identify $M=\nu_N,X=\E,\Phi=\operatorname{Id}$ and $\widetilde{X}=(h\circ\rho)\cdot\E$. By construction 
\[
\widetilde{\Phi}(\nu_N)=\{\widetilde{X}\neq 0\} = \{h\circ\rho > 0 \}.
\]

Let $c\in\R_{>0}$. If we show that there exists $c'\in\R_{>0}$ such that
\[
{\left(\rho\circ\widetilde{\Phi}^{-1}\right)}^{-1}(c) \subset
\rho^{-1}(c')
\]
then the result will follow by the fact that the intersection of both sets with a fiber is diffeomorphic to the sphere.

Let $s\in \R_{>0}$. Denote by $\exp_{h,s}(t)$ the solution to the ODE defined by the equations
\begin{align*}
\frac{d}{dt}&\exp_{h,s}(t) =h\left(s\cdot \exp_{h,s}^2(t)\right)\cdot \exp_{h,s}(t) \\
&\exp_{h,s}(0)=1.
\end{align*}
This solution exists for all $t\in \R$ because $\exp_{h,s}(t)$ is bounded by $\exp(t)$. Note that $\exp_{h,s}(t)$ is positive and monotone in $t$.

We claim that the flow $\Psi^t_X(v)$ of $X$ is given by
\[
\Psi^t_X(v) = m^{\exp_{h,\rho(v)}(t)}(v).
\]
Indeed, taking the derivative of the RHS by $t$ we get
\begin{align*}
\frac{d}{dt}\Big\lvert_{t_0}m_X^{\exp_{h,\rho(v)}(t)}(v) & = 
\frac{d}{dt}\Big\lvert_{\log\exp_{h,\rho(v)}(t_0)}m_X^{\exp(t)}(v) \cdot \frac{d}{dt}\Big\lvert_{t_0}\log\exp_{h,\rho(v)}(t) \\
& = \E\left(m_X^{\exp_{h,\rho(v)}(t_0)}(v)\right) \cdot \frac{\frac{d}{dt}\Big\lvert_{t_0}\exp_{h,\rho(v)}(t)}{\exp_{h,\rho(v)}(t_0)} \\
& = \E\left(m_X^{\exp_{h,\rho(v)}(t_0)}(v)\right) \cdot h\left(\rho(v)\cdot \exp_{h,\rho(v)}^2(t_0)\right) \\
& = \E\left(m_X^{\exp_{h,\rho(v)}(t_0)}(v)\right) \cdot h\left( \left(\rho \circ m_X^{\exp_{h,\rho(v)}(t_0)}\right)(v)\right) \\
& = \E\left(m_X^{\exp_{h,\rho(v)}(t_0)}(v)\right) \cdot \left(h\circ \rho\right)\left( m_X^{\exp_{h,\rho(v)}(t_0)}(v)\right) \\
& = \widetilde{X}\left(m_X^{\exp_{h,\rho(v)}(t_0)}(v)\right).
\end{align*}

Now let $v_0\in {\left(\rho \circ \widetilde{\Phi}^{-1}\right)}^{-1}(c)$ and denote $t_0 = \sqrt{\nicefrac{a}{2c}}$ and 
\[
v_1 = \left(m^{t_0}\circ \widetilde{\Phi}^{-1} \right)(v_0).
\] 

Since $\widetilde{X}=\E$ on $\{\rho < a\}$, the embedding $\widetilde{\Phi}\Big\lvert_{\{\rho < a\}}:\{\rho < a\} \to \nu_N$ is the identity on that set. By homogeneity 
\[
\rho(v_1) = \left (\rho\circ m^{\sqrt{\nicefrac{a}{2c}}}\right)\left(\widetilde{\Phi}^{-1}(v_0) \right) = \frac{a}{2}
\]
so $\widetilde{\Phi}$ is the identity on $v_1$. Using the fact that $m^{\exp(t)}\sim_{\widetilde{\Phi}} \Psi_X^t$ we get
\begin{align*}
v_1 = \left(\widetilde{\Phi}\circ m^{t_0}\circ \widetilde{\Phi}^{-1} \right)(v_0) = \Psi_X^{\log{t_0}}(v_0)
\end{align*}
and therefore
\begin{align*}
v_0 &= \Psi_X^{-\log{t_0}}(v_1) \\
&= m^{\exp_{h,\rho(v_1)}(-\log{t_0})}(v_1) \\
&= m^{\exp_{h,\nicefrac{a}{2}}(-\log{t_0})}(v_1).
\end{align*}

Finally write $c'={\exp_{h,\nicefrac{a}{2}}(-\log{t_0})}^2 \cdot \frac{a}{2}$. We have
\begin{align*}
\rho(v_0) &= \rho\left(m^{\exp_{h,\nicefrac{a}{2}}(-\log{t_0})}(v_1) \right) \\
&= {\exp_{h,\nicefrac{a}{2}}(-\log{t_0})}^2 \cdot \rho(v_1) \\
&= c'
\end{align*}
so 
\[
{\left(\rho\circ\widetilde{\Phi}^{-1}\right)}^{-1}(c) \subset
\rho^{-1}(c'). \qedhere
\]
\end{proof}

\section{Smoothly locally trivial stratified spaces}\label{stratified_section}
\begin{definition}
Let $C$ be a Hausdorff, second countable and paracompact space. A \textbf{stratification} of $C$ is a locally finite decomposition $\S$ of $C$ into a disjoint union of \textbf{strata} --- locally closed subsets each carrying a structure of a smooth manifold. In addition, this decomposition needs to satisfy
the \textit{condition of the frontier}, which asserts that the closure of a stratum is a union of strata.
\end{definition}
\begin{rem}
An equivalent definition of the condition of the frontier is the following: for $X,Y\in\S$, if $X\cap \overline{Y} \neq \emptyset$ then $X\subset \overline{Y}$.
\end{rem}
\begin{definition}\label{stratified_subspace}
A subset $\S'\subset \S$ such that
\[
C' = \bigsqcup_{X\in\S'}X
\]
is closed in $C$ is a \textbf{stratified subspace} of $(C,\S)$.
\end{definition}
\begin{definition}
Let $X,Y\in\S$ and define the relation $X<Y$ when $X\subset \overline{Y}$. One can check that this is an order relation.
\end{definition}
Let $M$ be a manifold of dimension $m$.
\begin{definition}
Let $C\subset M$ be a subset with a stratification $\S$ of $C$, such that the strata with their manifold structures are embedded submanifolds of $M$. Then $(C, \S)$ is a \textbf{stratified subset} of $M$.
\end{definition}
From now on we will mostly work with stratified subsets of manifolds.
In order to get nice properties of these stratification, such as $X<Y$ implies $\dim{X}<\dim{Y}$, one commonly adds additional regularity assumptions to the definition such as Whitney conditions (A) and (B) (see \cite{mather2012notes}). For our purposes we use the following stronger condition:
\begin{definition}\label{locally_trivial} Let $M$ be a manifold of dimension $n$. 
A stratified subset $C\subset M$ is \textbf{smoothly locally trivial with conical fibers} if for each $X\in\S$ of dimension $k$ and $p\in X$ there exist an open neighbourhood $U_p$ of $p$ and a diffeomorphism
\[
\theta_p:U_p\to \Omega_p \subset \R^k\times \R^{n-k}
\]
such that
\begin{enumerate}
    \item $\theta_p(X\cap U_p) = \left(\R^k\times \{0\}\right)\cap \Omega_p$.
    \item For each $Y\in\S$ with $X<Y$ there exists a subset $C_{Y,p}\subset \R^{n-k}\setminus\{0\}$ which is \textbf{conical}, i.e. invariant under multiplication by $t\in (0,1)$, such that
    \[
    \theta_p(Y\cap U_p) = \left(\R^k \times C_{Y,p}\right) \cap \Omega_p.
    \]
\end{enumerate}
\end{definition}
\begin{definition}\label{conical_neighbourhood}
A \textbf{conical neighbourhood} of $p$ is the data $(U_p, \Omega_p, \theta_p)$.
\end{definition}
\begin{rem}\label{also_whitney_b}
The local condition of Definition \ref{locally_trivial} implies Whitney (B) regularity. All Whitney (B) regular stratified subsets are \textit{topologically locally trivial} \cite[Corollary~10.6]{mather2012notes} but not necessarily smoothly locally trivial.
\end{rem}
\begin{example}[{\cite[Example~13.2]{Whitney1992}}]
Let $C\subset \R^2 \times (-1,1)$ be the zero set of
\[
f(x,y,t) = xy(y-x)(y - (3+t)x)
\]
stratified by splitting to the $t$ axis and connected components of the rest.

Then this stratification is Whitney (B) but not smoothly locally trivial.
\end{example}
\begin{example}
Let $M=\R^2$ and $N$ the zero set of $f(x,y)=y(y-x^2)$, which can be stratified by splitting to the origin and the rest. Then this stratification is Whitney (B) and smoothly locally trivial but not smoothly locally trivial with conical fibers.
\end{example}
\begin{definition}
Let $G$ be a Lie group acting on $M$. A stratification $\S$ of a closed subset $C\subset M$ is a \textbf{$G$-stratification} if each $X\in\S$ is $G$-invariant. 
\end{definition}
\begin{example}\label{orbit_types_example}
Let $G$ be a compact Lie group and let $M$ be a $G$-manifold. In Section \ref{orbit_type_section} we will describe two different definitions of stratifications of $M$ by orbit types and show they are $G$-stratifications that are smoothly locally trivial with conical fibers.
\end{example}
\begin{example}\label{momentum_zero_level_set}
Let $G$ be a compact Lie group and let $M$ be a Hamiltonian $G$-space with a $G$-equivariant momentum map
\[
\mu:M\to \g^*.
\]

Sjamaar and Lerman \cite[Theorem~2.1]{sjamaar1991stratified} proved that the intersection of $\mu^{-1}(0)$ with a stratum of one of the stratifications by orbit types of Example \ref{orbit_types_example} is a manifold, and that this induces a Whitney (B) regular $G$-stratification of $\mu^{-1}(0)$.
\end{example}
\begin{example}\label{norm_square_crit_set}
With the assumptions of Example \ref{momentum_zero_level_set}, fix an $Ad$-invariant inner product on the Lie algebra $\g$ of $G$ and fix the induced inner product on $\g^*$. 

In Subsection \ref{crit_set_subsection} we will show that the critical set $\text{Crit}\norm{\mu}^2$ can be decomposed into a $G$-stratification similarly to Example \ref{momentum_zero_level_set}. We will also show that this stratification is smoothly locally trivial with conical fibers and that $\mu^{-1}(0)$ is a union of connected components of it, and deduce that the $G$-stratification of $\mu^{-1}(0)$ is also smoothly locally trivial with conical fibers.
\end{example}

\subsection{Control data}\label{control_data_subsection}
\hfill

Let $(C,\S)$ be a stratified subset of a manifold $M$. Denote
\begin{align*}
\S_{<d} &= \{X\in \S: \dim{X}<d\} \\
C_{<d} &= \bigsqcup_{X\in\S_{<d}}X \\
M_d &= M - C_{<d}
\end{align*}
and note that a stratum of dimension $d$ is closed in $M_d$.

Let ${\{\left(V_X, T_X, m_X^t, \rho_X \right)\}}_{X\in\S}$ be a collection of the following data --- for each $X\in\S$ with $\dim{X}=d$:
\begin{enumerate}
    \item $V_X:M_d\to TM_d$ is a convenient Euler like vector along $X\subset M_d$, inducing a tubular neighbourhood embedding $\Phi_X:\nu_X\to M_d$.
    \item $T_X = \Phi_X(\nu_X)$.
    \item $m_X^t:T_X\to T_X$ is the multiplication by $t$ in $T_X$, identifying $T_X$ with $\nu_X$ using $\Phi_X$. Equivalently $m_X^{\exp(t)}:T_X\to T_X$ is the flow of $X$ restricted to $T_X$
    \item $\rho_X:T_X\to \R_{\geq 0}$ is the pullback by $\Phi^{-1}$ of a distance function on $\nu_X$.
\end{enumerate}
We call the data $\left(V_X, T_X, m_X^t, \rho_X \right)$ a tubular neighbourhood of $X$ and sometimes refer to it as just $T_X$. We call such a collection a \textbf{collection of tubular neighbourhood for $\S$}.

In this section we will give definitions of desired properties of such collection. The following is the topological part:
\begin{definition}\label{adjusted}
The collection ${\{\left(V_X, T_X, m_X^t, \rho_X \right)\}}_{X\in\S}$ is \textbf{adjusted with regard to $\S$} if the following properties hold:
\begin{enumerate}[label=(\subscript{AD}{{\arabic*}})]
    \item \label{ad1} $T_X\cap T_Y \neq \emptyset$ if and only if $X,Y$ are comparable.
    \item \label{ad2} $T_X\cap Y \neq \emptyset$ if and only if $X<Y$.
    \item \label{ad3} For $X,Y\in\S$ of dimension $d$, $\overline{T_X}\cap\overline{T_Y}\cap M_d = \emptyset$.
    \item \label{ad4} For $X\in\S$ with $\dim{X}=d$, the set $\rho_X^{-1}([0,a])$ is closed in $M_d$.
\end{enumerate}
\end{definition}
\begin{rem}
\ref{ad4} is already implied by the definition of $T_X$ and Proposition \ref{distance_close}. We include it in Definition \ref{adjusted} to give it a name and emphasize its importance for future constructions.
\end{rem}
From now on we assume that the collection is adjusted.
\begin{definition}\label{tangential}
Let $X\in\S$. We say $\left(V_X, T_X, m_X^t, \rho_X \right)$ is \textbf{tangent to higher strata} if for $Y>X$ and $y\in Y\cap T_X$ we have
\[
V_X(y)\in T_yY.
\]

The collection is \textbf{tangential with regard to $\S$} if each $\left(V_X, T_X, m_X^t, \rho_X \right)$ is tangent to higher strata.
\end{definition}
When the stratification is fixed we will often omit the phrase "with regard to $\S$" from Definitions \ref{adjusted} and \ref{tangential} and just say "adjusted" and "tangent to higher strata".
\begin{definition}\label{pre_commute}
Let $X,Y\in\S$.
We say that $T_X,T_Y$ \textbf{pre-commute} if either $X$ and $Y$ are not comparable or $X<Y$ and the following properties hold. For $t\geq 0$ and in the domain of definition of both sides:
\begin{enumerate}[label=(\subscript{PC}{{\arabic*}})]
    \item $m_X^0\circ m_Y^t = m_X^0$. Equivalently $Dm_X^0\cdot V_Y = 0$.
    \item $\rho_X\circ m_Y^t=\rho_X$. Equivalently $\L_{V_Y}\rho_X=0$.
\end{enumerate}

The collection is \textbf{pre-commutative} if for each $X,Y\in\S$, their tubular neighbourhood $T_X,T_Y$ pre-commute.
\end{definition}
\begin{rem}
An adjusted and pre-commutative collection has the properties of \textbf{control data} in the sense of Mather \cite{mather2012notes}. From now on, by control data we mean an adjusted and pre-commutative collection.
\end{rem}
\begin{definition}\label{commute}
Let $X,Y\in\S$ and assume $T_X,T_Y$ pre-commute and $V_X$ is tangent to $Y$.
We say that $T_X,T_Y$ \textbf{commute} if either $X$ and $Y$ are not comparable or $X<Y$ and the following properties hold. For $s>0,t\geq 0$ and in the domain of definition of both sides:
\begin{enumerate}[label=(\subscript{C}{{\arabic*}})]
    \item The maps $m_X^s,m_Y^t$ commute. Equivalently $[V_X,V_Y]=0$.
    \item $\rho_Y\circ m_X^s = \rho_Y$. Equivalently $\L_{V_X}\rho_Y=0$.
\end{enumerate}

The collection is \textbf{commutative} if for each $X,Y\in\S$, their tubular neighbourhood $T_X,T_Y$ commute.
\end{definition}
\begin{definition}\label{equivariant}
Let $G$ be a Lie group acting on $M$ and assume $\S$ is a $G$-stratification.
For $X\in\S$, we say $\left(V_X, T_X, m_X^t, \rho_X \right)$ is \textbf{$G$-equivariant} if $V_X$ is $G$-invariant (equivalently $m_X^t$ is $G$-equivariant) and $\rho_X$ is $G$-invariant.
\end{definition}

Given a smoothly locally trivial with conical fibers stratified subset $C\subset M$, in Section \ref{tangential_control_data_section} we will construct an adjusted, tangential and pre-commutative collection of tubular neighbourhoods for it --- tangential control data. In Section \ref{commutative_control_data_section} we will upgrade this control data to also be commutative.

\section{Tangential control data}\label{tangential_control_data_section}
Let $M$ be a manifold of dimension $n$ and let $(C,\S)$ be a stratified subset of it that is smoothly locally trivial with conical fibers. The following is the main theorem of this section:
\begin{thm}\label{tangential_control_data_thm}
There exists a collection of tubular neighbourhoods for $\S$ that is adjusted, tangential and pre-commutative.

If $G$ is a compact Lie group acting on $M$ and $\S$ is a $G$-stratification then the collection can be constructed to be $G$-equivariant.
\end{thm}
We call such a collection \textbf{tangential control data}. We continue with the following strategy, similar to the one appearing in \cite{mather2012notes,pflaum2017equivariant}:
\begin{enumerate}
    \item Use the local finiteness of $\S$, para-compactness and normality of $M$ to deal with the topological requirements (this part was implicit in \cite{mather2012notes,pflaum2017equivariant}).
    \item Solve locally using the local form from Definition \ref{locally_trivial}.
    \item Patch the local data to a global one.
\end{enumerate}

\subsection{Topological considerations}
\begin{prop}
Let $X$ be a stratum of dimension $d$. Then
\[
A_X := \bigsqcup_{\substack{
\dim{Y} \geq d \\
X\not< Y
}}Y
\subset M_d
\]
is a closed subset of $M_d$.
\end{prop}
\begin{proof}
We show that the complement is open. Let $p\in M_d - A_X$.

If $p\notin C$ this follows from the fact that $C$ is closed and $A_X\subset C$. 

Otherwise $p\in Y$ for some $Y\in\S$ which satisfies $X<Y$. Let $U_p\subset M_d$ be a conical neighbourhood of $p$, we show that $U_p \subset M_d - A_X$.

Let $z\in U_p$. If $z\notin C$ then $z\notin A_X$. Otherwise assume $z\in Z$ for some $Z\in\S$. Then by definition of a conical neighbourhood $Y<Z$ and since "$<$" is transitive we get that $X<Z$, so $z\notin A_X$.
\end{proof}
\begin{cor}\label{safe_neighbourhood_md}
Let $X$ be a stratum of dimension $d$. There exists a neighbourhood $U_X\subset M_d$ of $X$ with the following property. For each $Y\in\S$ with $\dim{Y}\geq d$ we have $\overline{U_X}\cap Y \cap M_d\neq\emptyset$ if and only if $X<Y$.
\end{cor}
\begin{proof}
Use the previous proposition, the fact that $X$ is closed in $M_d$ and normality of $M_d$.
\end{proof}
\begin{prop}\label{locally_finite_neighbourhood}
There exists a locally finite collection of open sets $\{O_X\}_{X\in\S}$ such that 
\[
\forall X\in\S,\: X\subset O_X
\]
\end{prop}
\begin{proof}
For each point $p\in M$ let $U_p$ be a neighbourhood of $p$ which meets only a finite number of strata. Using the fact that $M$ is paracompact and ${\{U_p\}}_{p\in M}$ is a cover of $M$ we get a locally finite open cover $\mathcal{U}$ of $M$ which is a refinement of ${\{U_p\}}_{p\in M}$ so that each $U\in\mathcal{U}$ intersects finitely many strata. For each $X\in\S$ define
\begin{align*}
\mathcal{U}_X &:= \{U\in\mathcal{U}: U\cap X\neq \emptyset\} \\
O_X &:= \bigcup_{U\in \mathcal{U}_X}U.
\end{align*}

We prove $\{O_X\}_{X\in\S}$ satisfies the required properties. $\mathcal{U}_X$ is a cover of $X$ so it is enough to show the collection is locally finite.

Let $p\in M$. By construction there exists a neighbourhood $p\in W_p$ such that $W_p$ intersects finitely many $U\in\mathcal{U}$, say ${\{U_i\}}_{i\leq r}\subset \mathcal{U}$. Each $U_i$ intersects finitely many strata and is thus contained in finitely many ${\{O_{X_{i,j}}\}}_{i\leq r, j\leq m_i}$.  It follows that $W_p$ intersects only the finite set ${\{O_{X_{i,j}}\}}_{i\leq r, j\leq m_i}$.
\end{proof}
\begin{lem}{\label{good_neighbourhood}}
There exists a collection of open sets $\{W_X\}_{X\in\S}$ satisfying the following properties:
\begin{enumerate}
    \item \label{wx_in_md} For $X\in\S$ of dimension $d$ we have $X\subset W_X \subset M_d$.
    \item The collection is locally finite.
    \item $W_X\cap W_Y \neq \emptyset$ if and only if $X,Y$ are comparable.
    \item \label{stronger_ad2} For $X\in\S$ of dimension $d$ we have $\overline{W_X}\cap Y \cap M_d \neq\emptyset$ if and only if $X<Y$.
    \item For strata $X\neq Y$, both of dimension $d$, we have $\overline{W_X}\cap\overline{W_Y}\cap M_d = \emptyset$.
\end{enumerate}
\end{lem}
\begin{rem}
Properties (\ref{wx_in_md}) and (\ref{stronger_ad2}) imply that $W_X\cap Y \neq \emptyset$ if and only if $X<Y$.
\end{rem}
\begin{proof}[Proof of Lemma \ref{good_neighbourhood}]
Let ${\{O_X\}}_{X\in\S}$ be a locally finite collection from Proposition \ref{locally_finite_neighbourhood} and let ${\{U_X\}}_{X\in\S}$ be a collection of neighbourhoods as in Corollary \ref{safe_neighbourhood_md}.

${\{O_X\cap U_X\}}_{X\in\S}$ is locally finite and therefore for every subset $J\subset \S$, the set
\[
\bigcup_{Y\in J}\overline{O_Y\cap U_Y}
\]
is closed in $M$. 

Let $X\in\S$ of dimension $d$. The closed set
\[
B_X:= \bigcup_{\substack{
\dim{Y}\leq d \\
Y\not< X
}}
\overline{O_Y\cap U_Y}
\]
does not intersect $X$ by the properties of ${\{U_Y\}}_{Y\in\S}$. Since $X$ and $B_X\cap M_d$ are closed in $M_d$ we can use normality to find a neighbourhood $W_X\subset M_d$ of $X$ such that 
\begin{align*}
    & W_X\subset O_X\cap U_X \cap M_d \\
    & \overline{W_X}\cap B_X\cap M_d = \emptyset .
\end{align*}

We claim ${\{W_X\}}_{X\in S}$ satisfies the required properties. Indeed:
\begin{enumerate}
    \item By construction.
    \item $W_X\subset O_X$ and ${\{O_X\}}_{X\in\S}$ is locally finite.
    \item Assume $W_X\cap W_Y\neq \emptyset$, without loss of generality $\dim{Y}\leq \dim{X}=d$. Since $W_X\subset M_d$ we have
\[
\overline{W_X}\cap W_Y \cap M_d \neq \emptyset
\]
so $W_Y\subset O_Y\cap U_Y$ cannot be a subset of $B_X$. By definition of $B_X$ we get $Y<X$.
    \item This follows from the same property of ${\{U_X\}}_{X\in\S}$.
    \item Let $X,Y\in\S$ of dimension $d$. They are not comparable and therefore $\overline{W_Y}\subset B_X$ so $\overline{W_X}\cap\overline{W_Y}\cap M_d = \emptyset$.
\end{enumerate}
\end{proof}

\subsection{Local tangential control data}
\begin{definition}
Let $X\in\S$ be a stratum of dimension $d$, let $p\in X$ and let $\left(U_p, \Omega_p, \theta_p\right)$ be a conical neighbourhood of $p$. Denote by $(x_1,\dots,x_d,y_1,\dots,y_{n-d})$ the coordinates in $\Omega_p$ such that $\theta_p(X\cap U_p)= \{y=0\}\cap \Omega_p$. 

The \textbf{Euler-like vector field associated to $\left(U_p, \Omega_p, \theta_p\right)$ in $\Omega_p$} is
\[
\sum_{i=1}^{n-d}y_i\frac{\partial}{\partial y_i}
\]
and the \textbf{Euler-like vector field associated to $\left(U_p, \Omega_p, \theta_p\right)$ in $U_p$} is its pullback by $\theta_p$.
\end{definition}
\begin{rem}
By the definition of a conical neighbourhood, the Euler-like vector field associated to $\left(U_p, \Omega_p, \theta_p\right)$ in $U_p$ is tangent to higher strata in $U_p$ and is Euler-like for $X\cap U_p$.
\end{rem}
\begin{lem}\label{local_submersion}
Let $X\in\S$ with $\dim{X}=d$, $p\in X$ and $U$ a neighbourhood of $p$.
Let $L^l$ be a manifold and $f:U\rightarrow L^l$ a smooth map such that $f\lvert_{X\cap U}$ is a submersion. 

Then there exists a vector field $V_p$ defined on a neighbourhood $W_p\subset U$ of $p$ such that:
\begin{enumerate}
    \item $V_p$ is Euler-like along $X\cap W_p$.
    \item $V_p$ is tangent to higher strata in $W_p$.
    \item $V_p$ is tangent to level sets of $f$.
\end{enumerate}
\end{lem}
\begin{proof}
The claim is local so we can take a conical neighbourhood $U_p$ of $p\in X$ and identify $U_p\subset U$ with $\Omega_p \subset \R^d\times \R^{n-d}$ with coordinates $(x,y)$. Here $X\cap U_p=\{y=0\}\cap U_p$ and each $Y\in\S$ with $X<Y$ is of the form $Y\cap U_p=\left(\R^k \times C_{Y,p}\right) \cap U_p $ for some conical subset $C_{Y,p}\subset \R^{n-d} \setminus \{0 \}$. We can further assume $f$ takes values in $\R^l$ by taking a smaller neighbourhood.

Since $f\lvert_X$ is a submersion at $p$, the differential
\[
Df = \begin{pmatrix}
Df_1\lvert_{TX} \\
\vdots \\
Df_l\lvert_{TX}
\end{pmatrix}
\]
has rank $l$ at $p$ and by permuting the $x$ coordinates we can arrange the top $l\times l$ block
\[
{\Bigl\{\frac{\partial f_i}{\partial x_j}\Bigr\}}_{i,j\leq l} = \begin{pmatrix}
\frac{\partial f_1}{\partial x_1} & \dots & \frac{\partial f_1}{\partial x_l} \\
\vdots & \ddots & \vdots \\
\frac{\partial f_l}{\partial x_1} & \dots & \frac{\partial f_l}{\partial x_l}
\end{pmatrix}
\]
to be invertible at $p$. Now we define $\Psi:U\rightarrow \R^d\times\R^{n-d}$ by
\[
\Psi(x_1,\dots,x_d,y_{1},\dots,y_{n-d}) = 
\begin{pmatrix}
f_1(x_1,\dots,x_d,y_{1},\dots,y_{n-d}) \\
\vdots \\
f_l(x_1,\dots,x_d,y_{1},\dots,y_{n-d}) \\
x_{l+1} \\
\vdots \\
x_d \\
y_1 \\
\vdots \\
y_{n-d}
\end{pmatrix}
\]
with differential
\[
\begin{pmatrix}
\frac{\partial f_1}{\partial x_1} & \dots & \frac{\partial f_1}{\partial x_l} & \frac{\partial f_1}{\partial x_{l+1}} & \dots & \frac{\partial f_1}{\partial y_{n-d}}\\
\vdots & \ddots & \vdots & \vdots & \ddots & \vdots \\
\frac{\partial f_l}{\partial x_1} & \dots & \frac{\partial f_l}{\partial x_l} & \frac{\partial f_l}{\partial x_{l+1}} & \dots & \frac{\partial f_l}{\partial y_{n-d}} \\
& & & 1 & & \\
& & & & \ddots & \\
& & & & & 1
\end{pmatrix}
\]
which by construction is invertible at $p$. Using the inverse function theorem we can find a neighbourhood $W_p \subset U$ of $p$ such that $\Psi$ is a diffeomorphism on $W_p$. 

We claim that $\left( W_p, \Psi(W_p), \Psi \right)$ is a conical neighbourhood of $p$. 
Indeed, $\Psi$ preserves the last $n-l$ coordinates and in particular the $y$ coordinates. We get that
\[
\Psi(X\cap W_p) = \left(\R^d \times \{0\}\right) \cap \Psi(W_p).
\]
For $Y\in\S$ with $X<Y$, by assumption 
\[
Y\cap W_p = (\R^d\times C_{Y,p})\cap W_p
\]
for a conical $C_{Y,p}$. Using the fact that $\Psi$ preserves the $y$ coordinates we obtain
\[
\Psi(Y\cap W_p) = (\R^d\times C_{Y,p})\cap \Psi(W_p).
\]

Notice that $f\circ \Psi^{-1}:\Psi(W_p)\rightarrow \R^l$ is the projection to the first $l$ coordinates in $\Psi(W_p)$. 
This implies that the Euler-like vector field $V_p'$ associated to $\left( W_p, \Psi(W_p), \Psi \right)$ in $\Psi(W_p)$ satisfies 
\[
D(f\circ \Psi^{-1})\cdot V_p'=0
\]
and therefore the Euler-like vector $V_p$ associated to $\left( W_p, \Psi(W_p), \Psi \right)$ in $W_p$ satisfies
\[
Df\cdot V_p = 0.
\]

Finally, the Euler-like vector $V_p$ is tangent to higher strata in $W_p$ because $V_p'$ is tangent to higher strata in $\Psi(W_p)$.
\end{proof}
\begin{lem}\label{only_top_stratum}
Let $X$ be a stratum, $p\in X$ and $U$ a neighbourhood of $p$. Let $Y,Z\in\S$ with $Z<Y<X$ and let $T_Y, T_Z$ be tubular neighbourhoods of $Y,Z$ such that $U\subset T_Y\cap T_Z$ and $T_Y, T_Z$ pre-commute.

Let $V_X$ be a vector field defined on $U$ that is Euler-like along $X\cap U$ such that
\[
D\left(m^0_Y, \rho_Y \right)\cdot V_X = 0.
\]

Then
\[
D\left(m^0_Z, \rho_Z \right)\cdot V_X = 0.
\]
\end{lem}
\begin{proof}
By the assumption that $T_Y,T_Z$ pre-commute we get 
\begin{align*}
m^0_Z\circ m^0_Y = m^0_Z \\
\rho_Z\circ m^0_Y = \rho_Z.
\end{align*}
and therefore
\begin{align*}
D\left(m^0_Z, \rho_Z \right)\cdot V_X &= D\left(\left(m^0_Z, \rho_Z \right)\circ m^0_Y\right)\cdot V_X \\
&= D\left(m^0_Z, \rho_Z \right)\cdot \underbrace{Dm^0_Y\cdot V_X}_0. \\
&= 0.
\end{align*}
\end{proof}
\begin{lem}\label{submersion_near}
Let $X\in\S$ with $\dim{X}=d$ and let $V_X$ be an Euler-like vector field along $X$, defined on $M_d$, that induces a tubular neighbourhood $(U,\Phi)$. 

Then there exists a neighbourhood $W\subset \Phi(U)$ of $X$ such that for each $Y>X$, the map
\[
\left( m^0\circ \Phi^{-1} \right)\Big\lvert_{W\cap Y}
\]
is a submersion.
\end{lem}
\begin{proof}
Let ${\{Y_i\}}_{i\in I}$ be the set of strata $Y$ such that $X<Y$. By the assumption that $\S$ is locally finite and the condition of the frontier this set is finite.

For each $Y_i$, use Remark \ref{also_whitney_b} and \cite[Lemma~7.3]{mather2012notes} to obtain a neighbourhood $U_i\subset \Phi(U)$ such that
\[
\left( m^0\circ \Phi^{-1} \right)\Big\lvert_{U_i\cap Y}
\]
is a submersion. Now take
\[
W = \bigcap_{i\in I}U_i.
\]
\end{proof}
\begin{lem}\label{distance_submersion_too}
Let $X<Y$ and let $V_X$ be an Euler-like vector field along $X$ tangent to $Y$, inducing a tubular neighbourhood $(U, \Phi)$. Assume that there exists a neighbourhood $W$ of $X$ such that
\[
\left( m^0\circ \Phi^{-1} \right)\Big\lvert_{W\cap Y}
\]
is a submersion. Then for any distance function $\rho:\nu_X\to\R_{\geq 0}$ the map
\[
\left( m^0\circ \Phi^{-1}, \rho\circ \Phi^{-1}\right)\Big\lvert_{W\cap Y}
\]
is a submersion.
\end{lem}
\begin{proof} Follows from the assumptions and the fact that $V_X$ is a direction in $Y$ that is in $\ker D\left(m^0\circ \Phi^{-1}\right)$ and not in $\ker D \left(\rho\circ \Phi^{-1}\right)$.
\end{proof}
\subsection{Proof of Theorem \ref{tangential_control_data_thm}}
\hfill

Let ${\{W_X\}}_{X\in\S}$ be the collection of open sets from Lemma \ref{good_neighbourhood}. Let $G$ be a compact Lie group acting on $M$ and assume that $\S$ is a $G$-stratification.

We use ascending induction on the dimension of strata $d$.

\underline{Assumption for construction step $d$}:
There exists a collection ${\{\left(V_X, T_X, m_X^t, \rho_X \right)\}}_{\dim{X}<d}$ of convenient tubular neighbourhoods of strata of dimension $<d$ which satisfies the following properties:
\begin{enumerate}
    \item\label{in_proof_equivariant} The collection is $G$-equivariant.
    \item\label{in_proof_adjusted} $T_X\subset W_X$.
    \item\label{in_proof_tangential} Each $V_X$ in the collection is tangent to higher strata (not only of dimension $<d$).
    \item\label{in_proof_submersion} For $X\in \S$ with $\dim{X}<d$ and $Y>X$, the restricted map $(m^0_X, \rho_X)\lvert_{T_X\cap Y}$ is a submersion.
    \item\label{in_proof_pre_commute} each pair $T_X,T_Y$ in the collection pre-commute.
\end{enumerate}
Note that Property (\ref{in_proof_adjusted}) and the fact we work with convenient tubular neighbourhoods imply that ${\{\left(V_X, T_X, m_X^t, \rho_X \right)\}}_{\dim{X}<d}$ is adjusted with regard to $\S_{<d}$.

\underline{Construction step $0$}: Let $X\in\S$ with $\dim{X}=0$ i.e. $X={\{p_j\}}_{j\in J}$ a discrete set. For each $j\in J$ let $(U_{p_j}, \Omega_{p_j}, \theta_{p_j})$ be a conical neighbourhood of $p_j$ such that $U_{p_j}\subset W_X$. Define the vector field $V_{X,p_j}$ to be the Euler-like vector field associated to $(U_{p_j}, \Omega_{p_j}, \theta_{p_j})$ in $U_{p_j}$.

Patch ${\{V_{X,p_j}\}}_{j\in J}$ to a vector field $V_X'$ on $M$ using a partition of unity. $V_X'$ is Euler-like along $X$ and tangent to higher strata. Since each $Y\in\S$ is $G$-invariant, the vector field $V_X''$ obtained by averaging $V_X'$ over $G$ also has these properties.

Now use Lemma \ref{submersion_near} with the Euler-like vector field $V_X''$ to obtain a neighbourhood $W_X'$ of $X$ such that 
\[
\left( m^0\circ \Phi^{-1} \right)\Big\lvert_{Y\cap W_X'}
\]
is a submersion for each $Y>X$. By the compactness of $G$ we can find a smaller neighbourhood $W_X''\subset W_X'$ of $X$ that is also $G$-invariant. 

Finally, use Lemma \ref{convenient_euler} to obtain a $G$-invariant bump function $\varphi:M\to [0,1]$ supported in $W_X''$ such that
\[
V_X\coloneqq \varphi\cdot V_X''
\]
is a $G$-invariant convenient Euler-like vector field along $X$, inducing a $G$-equivariant tubular neighbourhood embedding $\Phi_X$. Denote
\begin{align*}
T_X &= \Phi_X(\nu_X) \\
m_X^t &= m^t\circ \Phi_X^{-1}
\end{align*}
and take $\rho_X$ to be the pullback by $\Phi^{-1}$ of any $G$-invariant distance function on $\nu_X$.

Repeating the above construction for each $X\in\S$ with $\dim{X}=0$ we obtain the collection ${\{\left(V_X, T_X, m_X^t, \rho_X \right)\}}_{\dim{X}=0}$ which satisfies the construction assumption for step $d=1$:
\begin{enumerate}
    \item By construction.
    \item By construction.
    \item By construction.
    \item $m_X^0\lvert_Y$ is a submersion for each $Y>X$ by construction. Therefore this follows from Property (\ref{in_proof_tangential}) and Lemma \ref{distance_submersion_too}.
    \item Strata of the same dimension are not comparable.
\end{enumerate}

\underline{Construction step $d$}:
We build $T_X$ separately for each stratum $X\in\S$ with $\dim{X}=d$ by building $T_{X,p}$ locally, patching by a partition of unity and shrinking using Lemma \ref{convenient_euler}. We will also shrink all  neighbourhoods in the existing collection.

Let $X\in\S$ with $\dim{X}=d$ and let $p\in X$.
Then $p\in \overline{T_{Y_i}}$ for finitely many smaller strata ${\{Y_i\}}_{i\in I\subset \N}$ with $\dim{Y_i}=i$ and $i<d$. That each dimension has at most a single stratum follows from the definition of $\{{W_Y\}}_{Y\in\S}$ and Property (\ref{in_proof_adjusted}). Denote 
\[
I_p = \{i_1,\dots,i_s\}\sqcup \{k_1,\dots,k_r\}
\]
where the $i_j$ indexes are those such that $p\in T_{Y_{i_j}}$, the $k_j$ indexes are those such that $p\in \partial{T_{Y_{k_j}}}$ and both sets are ordered by the dimension of the corresponding stratum. By Property (\ref{in_proof_adjusted}):
\[
Y_{i_1} < \dots < Y_{i_s}.
\]

Define 
\begin{align}
\label{in_proof_Ap} A_p &= \bigcup_{1\leq j \leq r}\{\rho_{Y_{k_j}}\leq 1\} \\
\label{in_proof_Bp} B_p &= \bigcup_{\substack{
Y\in S_{<d} \\
p\notin \overline{T_Y}}}\overline{T_Y}.
\end{align}
Both sets don't contain $p$ and are closed in $M_d$. For $A_p$ this follows from Proposition \ref{distance_close} and for $B_p$ from Property (\ref{in_proof_adjusted}).
Denote 
\[
W_p := W_X \cap (M_d\setminus A_p) \cap (M_d\setminus B_p) \cap \bigcap_{1\leq j\leq s}T_{Y_{i_j}}
\]
this is a neighbourhood of $p$ on which $(m^0_{Y_{i_j}}, \rho_{Y_{i_j}})$ is defined or all $1\leq j \leq s$.

By Property (\ref{in_proof_equivariant}), for each $g\in G$ the sets $I_{g\cdot p},A_{g\cdot p}, B_{g\cdot p}$ and $W_{g\cdot p}$ are the same as those of $p$. Therefore we can denote them by $I_{G\cdot p},A_{G\cdot p}, B_{G\cdot p}$ and $W_{G\cdot p}$ respectively.

For each $g\in G$, by Property (\ref{in_proof_submersion}) and Lemma \ref{local_submersion} there exist a vector field $V_{g\cdot p}$, defined on a smaller neighbourhood $U_{g\cdot p}\subset W_{G\cdot p}$ of $g\cdot p$, such that:
\begin{itemize}
    \item $V_{g\cdot p}$ is Euler-like along $X\cap U_{g\cdot p}$.
    \item $V_{g\cdot p}$ is Tangent to higher strata inside $U_{g\cdot p}$.
    \item $V_{g\cdot p}$ satisfies
    \[
    D\left(m^0_{Y_{i_s}}, \rho_{Y_{i_s}}\right)\cdot V_{g\cdot p}=0.
    \]
    Moreover, the last equality and Lemma \ref{only_top_stratum} imply that
    \[
    \forall 1\leq j \leq s: \: D\left(m^0_{Y_{i_j}}, \rho_{Y_{i_j}}\right)\cdot V_{g\cdot p}=0. 
    \]
\end{itemize}

$G\cdot p$ is compact and ${\{U_{g\cdot p}\}}_{g\in G}$ is a cover of it in $W_{G\cdot p}$. Take a finite sub-cover ${\{U_{g_l\cdot p}\}}$, the union $U_{G\cdot p} = \bigcup_l U_{g_l\cdot p}$ and a partition of unity $\varphi_{l,G\cdot p}$ of $U_{G\cdot p}$ subordinate to it. The following vector field
\[
V_{G\cdot p} \coloneqq \sum_l \varphi_{l,G\cdot p}\cdot V_{g_l\cdot p}
\]
is Euler-like along $X\cap U_{G\cdot p}$, tangent to higher strata inside $U_{G\cdot p}$ and satisfies
\begin{equation}\label{eq_tangent_to_level_sets}
\forall 1\leq j \leq s: \: D\left(m^0_{Y_{i_j}}, \rho_{Y_{i_j}}\right)\cdot V_{G\cdot p}=0. 
\end{equation}

Now take $U_{G\cdot p}'\subset U_{G\cdot p}$ a smaller, $G$-invariant neighbourhood of $G\cdot p$ and $V_{G\cdot p}'$ the vector field obtained by averaging $V_{G\cdot p}\lvert_{U_{G\cdot p}'}$ over $G$. The vector field $V_{G\cdot p}'$ is $G$-invariant, Euler-like along $X\cap U_{G\cdot p}'$, tangent to higher strata inside $U_{G\cdot p}'$ and still satisfies Equation (\ref{eq_tangent_to_level_sets}).

$X$ is closed and $G$-invariant in $M_d$ and therefore the set
\[
{\{U_{G\cdot p}'\}}_{p\in X} \cup \{M_d\setminus X\}
\]
is a $G$-invariant open cover of $M_d$. Let ${\{\varphi_{G\cdot p_j}\}}_{j\in J}\cup \{ \varphi_{M\setminus X}\}$ be a $G$-invariant partition of unity subordinate to it and define
\[
V_X' = \sum_{j\in J}\varphi_{G\cdot p_j}\cdot V_{G\cdot p_j}'.
\]
This vector field is defined on $M_d$, $G$-invariant, Euler-like along $X$ and tangent to higher strata.

Similarly to step $0$, we use Lemma \ref{submersion_near} with the Euler-like vector field $V_X'$ and the compactness of $G$ to obtain a $G$-invariant neighbourhood $W_X'$ of $X$ such that 
\[
\left( m^0\circ \Phi^{-1} \right)\Big\lvert_{Y\cap W_X'}
\]
is a submersion for each $Y>X$. Now use Lemma \ref{convenient_euler} to obtain a $G$-invariant bump function $\varphi_X:M_d\to [0,1]$ supported in $W_X'$ such that
\[
V_X\coloneqq \varphi_X\cdot V_X':M_d \to TM_d
\]
is a $G$-invariant convenient Euler-like vector field along $X$, inducing a $G$-equivariant tubular neighbourhood embedding $\Phi_X:\nu_X\to M_d$. Denote
\begin{align*}
T_X &= \Phi_X(\nu_X) \\
m_X^t &= \Phi_X\circ m^t\circ \Phi_X^{-1}
\end{align*}
and take $\rho_X:T_X\to \R_{\geq 0}$ to be the pullback by $\Phi^{-1}$ of any $G$-invariant distance function on $\nu_X$.

\begin{claim_nn}
Adding ${\{\left(V_X, T_X, m_X^t, \rho_X \right)\}}_{\dim{X}=d}$ to the existing collection ${\{T_Y\}}_{\dim{Y}<d}$ yields a collection satisfying Properties (\ref{in_proof_equivariant}), (\ref{in_proof_adjusted}), (\ref{in_proof_tangential}) and (\ref{in_proof_submersion}).
\end{claim_nn}
Indeed:
\begin{enumerate}
    \item By the induction assumption and construction.
    \item By the induction assumption and construction.
    \item By the induction assumption and construction.
    \item By the induction assumption, construction and Lemma \ref{distance_submersion_too}.
\end{enumerate}

Note that property (5) is \textbf{not} necessarily satisfied --- let $v\in T_X\cap T_Y$ with $Y<X$ and $\dim{X}=d$. Then
\[
V_X(v) = \varphi_X(v)\cdot\sum_{j=1}^r\varphi_{G\cdot p_j}(v)\cdot V_{G\cdot p_j}'(v)
\]
and it is possible that some $p_j\in \partial T_Y$, so that $D\left(m^0_Y, \rho_Y\right)\cdot V_{G\cdot p_j}'$ is not necessarily $0$.

To solve this problem we shrink all tubular neighbourhoods of strata of dimension $<d$ using $h\coloneqq h_{\nicefrac{1}{3},\nicefrac{2}{3}}$ (Notation \ref{bump_function}). Denote by
\[
{\Bigl\{\left( \widetilde{V_X}, \widetilde{T_X}, \widetilde{m_X^t}, \widetilde{\rho_X} ,\right)\Bigr\}}_{\dim{X}< d+1}
\]
the collection obtained by such a shrinking and note that Properties (\ref{in_proof_equivariant}), (\ref{in_proof_adjusted}), (\ref{in_proof_tangential}) and (\ref{in_proof_submersion}) are preserved under such shrinking.
\begin{claim_nn}
The above collection also satisfies Property (\ref{in_proof_pre_commute}).
\end{claim_nn}
Let $Y<X$ be strata of dimension $< d+1$. We show that $\widetilde{T_X},\widetilde{T_Y}$ pre-commute.

If $\dim{X}<d$, then by the induction assumption the neighbourhoods before the shrinking $T_X,T_Y$ pre-commute. Therefore $V_X\lvert_{T_Y}$ is tangent to level sets of $\left(m_Y^0 , \rho_Y \right)$.
By Lemma \ref{shrinking_distance_level_sets} the level sets of 
$\left(\widetilde{m_Y^0} , \widetilde{\rho_Y} \right)$
are exactly those of 
$\left(m_Y^0 , \rho_Y \right)\big\lvert_{\widetilde{T_Y}}$
and since $\widetilde{V_X}$ is some function times $V_X$ we get that $\widetilde{V_X}\lvert_{\widetilde{T_Y}}$ is tangent to level sets of 
$\left(\widetilde{m_Y^0} , \widetilde{\rho_Y} \right)$.

If $\dim{X}=d$ then $\widetilde{T_X}=T_X$ and $\widetilde{T_Y}$ is the shrinking of $T_Y$ by $h$. Let $v\in \widetilde{T_X}\cap \widetilde{T_Y}$, then
\[
\widetilde{V_X}(v) = \varphi_X(v)\cdot\sum_{j=1}^r\varphi_{G\cdot p_j}(v)\cdot V_{G\cdot p_j}'(v)
\]
with $p_j \in X$ and $\varphi_{G\cdot p_j}(v)\neq 0$ for $1\leq j\leq r$. If we show that $p_j\in T_Y$ for all $1\leq j \leq r$ then by the local construction we are done.

Since $\varphi_{G\cdot p_j}(v)\neq 0$ we get that
\[
v\in U_{G\cdot p_j}' \subset (M_d\setminus A_{G\cdot p_j}) \cap (M_d\setminus B_{G\cdot p_j}).
\]
Now if $p_j\notin T_Y$ there are 2 cases, both of which lead to contradiction:
\begin{enumerate}
    \item $p_j\notin \overline{T_Y}$. But then by the definition (\ref{in_proof_Bp}) of $B_{G\cdot p_j}$ we get
    \[
    v\in \widetilde{T_Y}\subset \overline{T_Y}\subset B_{G\cdot p_j}
    \]
    which implies $v\notin U_{G\cdot p_j}'$.
    \item $p_j\in \partial{T_Y}$. But then by the definition (\ref{in_proof_Ap}) of $A_{G\cdot p_j}$ we get
    \[
    v\in \widetilde{T_Y}\subset \{ \rho_Y \leq 1\} \subset A_{G\cdot p_j}
    \]
    which implies $v\notin U_{G\cdot p_j}'$.
\end{enumerate}

This finishes the construction of step $d$ and thus the proof. \qed

\section{Commutative control data}\label{commutative_control_data_section}
\begin{thm}\label{commutative_control_data_theorem}
Let $M$ be a manifold and let $(C,\S)$ be a stratified subset of it. Assume there exist tangential control data ${\{ \left(T_X, \rho_X, m_X^t, V_X \right)\}}_{X\in\S}$.

Then there exists a modified collection of tubular neighbourhoods that remains adjusted, tangential and pre-commutative and is also commutative.

If $G$ is a compact Lie group acting on $M$, the stratification $\S$ is a $G$-stratification and the control data is $G$-equivariant then the modified collection is $G$-equivariant.
\end{thm}

We call such a collection \textbf{commutative tangential control data}. For the construction of it we will use an inductive process similar to the one used by Goresky in the construction of \textit{family of lines} in \cite{goresky_thesis,goresky_triangulation}. Our main technical tool in the induction step is a notion of "conjugating" $T_Y$ by $T_X$ for a pair of strata $X<Y$.

\subsection{Conjugating tubular neighbourhoods}
\begin{lem}\label{conjugating_by_some_func}
Let $X<Y$ be strata with $\dim{Y}=d$. Let $T_X,T_Y$ be convenient tubular neighbourhoods of $X,Y$ such that $V_X$ is tangent to $Y$ and such that $T_X,T_Y$ pre-commute. Let $f:\R_{>0}\rightarrow \R_{>0}$ be any smooth map that is $\equiv 1$ on $[1,\infty)$. 

Then the map
\begin{equation}\label{edited_euler_flow}
    \widetilde{m}_{Y}^{\exp(t)}(v) = \left( m_X^{(f\circ \rho_X)(v)} \circ m_Y^{\exp(t)} \circ  m_X^{{(f\circ \rho_X)}^{-1}(v)} \right)(v)
\end{equation}
is well defined on $M_d$ and is the flow of
\begin{equation}\label{edited_euler_vector}
\widetilde{V_Y}(v) := {Dm_X^{(f\circ \rho_X)(v)}}_{m_X^{{(f\circ \rho_X)}^{-1}(v)}(v)} \cdot V_Y\left( m_X^{{(f\circ \rho_X)}^{-1}(v)}(v)\right).
\end{equation}
Moreover, $\widetilde{V_Y}(v)$ is a convenient Euler-like vector field along $Y$.
\end{lem}
\begin{proof}
First, since $T_X$ is a convenient tubular neighbourhood of $X$, the function $f\circ \rho_X$ can be smoothly extended to $1$ on $M_d$ so that $m_X^{(f\circ \rho_X)(v)}$ is well defined. The map $m_Y^{\exp(t)}$ is the flow of the complete vector field $V_Y$ so it is defined for all $t$, so Equation (\ref{edited_euler_flow}) is well defined.

We show (\ref{edited_euler_flow}) is the flow of (\ref{edited_euler_vector}). By the uniqueness of the integral flow of a vector field it is enough to show that
\begin{align*}
& \frac{d}{dt}\Big\lvert_{t_0} \left( m_X^{(f\circ \rho_X)(v)} \circ m_Y^{\exp(t)} \circ  m_X^{{(f\circ \rho_X)}^{-1}(v)} \right)(v) \\
=& \widetilde{V_Y}\left(\left( m_X^{(f\circ \rho_X)(v)} \circ m_Y^{\exp(t_0)} \circ  m_X^{{(f\circ \rho_X)}^{-1}(v)} \right)(v)\right).
\end{align*}

Note that by pre-commutativity
\begin{align*}
& \left(\rho_X \circ m_X^{(f\circ \rho_X)(v)} \circ m_Y^{\exp(t)} \circ  m_X^{{(f\circ \rho_X)}^{-1}(v)} \right)(v) \\
=& {(f\circ \rho_X)(v)}^2\cdot \left(\rho_X \circ m_Y^{\exp(t)} \circ  m_X^{{(f\circ \rho_X)}^{-1}(v)} \right)(v) \\
= & {(f\circ \rho_X)(v)}^2\cdot \left(\rho_X \circ  m_X^{{(f\circ \rho_X)}^{-1}(v)} \right)(v) \\
= & \rho_X(v)
\end{align*}
which implies 
\[
\left(f \circ \rho_X \circ m_X^{(f\circ \rho_X)(v)} \circ m_Y^{\exp(t_0)} \circ  m_X^{{(f\circ \rho_X)}^{-1}(v)} \right) = (f\circ \rho_X).
\]

Now using the fact the $m_Y^{\exp(t)}$ is the flow of $V_Y$ we obtain
\begin{align*}
& \widetilde{V_Y}\left(\left( m_X^{(f\circ \rho_X)(v)} \circ m_Y^{\exp(t_0)} \circ  m_X^{{(f\circ \rho_X)}^{-1}(v)} \right)(v)\right) \\
=& {Dm_X^{(f\circ \rho_X)(v)}}_{\left(m_X^{{(f\circ \rho_X)}^{-1}(v)}\circ m_X^{(f\circ \rho_X)(v)} \circ m_Y^{\exp(t_0)} \circ  m_X^{{(f\circ \rho_X)}^{-1}(v)}\right)(v)} \\
& \cdot V_Y\left( \left(m_X^{{(f\circ \rho_X)}^{-1}(v)}\circ m_X^{(f\circ \rho_X)(v)} \circ m_Y^{\exp(t_0)} \circ  m_X^{{(f\circ \rho_X)}^{-1}(v)}\right)(v)\right) \\
= & {Dm_X^{(f\circ \rho_X)(v)}}_{(m_Y^{\exp(t_0)} \circ m_X^{{(f\circ \rho_X)}^{-1}(v)})(v)}\cdot V_Y\left( (m_Y^{\exp(t_0)} \circ m_X^{{(f\circ \rho_X)}^{-1}(v)})(v)\right) \\
= & \frac{d}{dt}\Big\lvert_{t_0} \left( m_X^{(f\circ \rho_X)(v)} \circ m_Y^{\exp(t)} \circ  m_X^{{(f\circ \rho_X)}^{-1}(v)} \right)(v).
\end{align*}

\begin{claim_nn}
$\widetilde{V_Y}$ is Euler-like along $Y$.
\end{claim_nn}
We use the equivalent property to being Euler like from Remark \ref{euler_like_deriv}. Let $g\in C^\infty(M_d)$ that vanishes on $Y$, we want to show that
\[
g - \L_{\widetilde{V_Y}}g
\]
vanishes to order 2 on $Y$.

For all $s\in\R_{>0}$, since $m_X^s$ preserves $Y$ we get that $(g\circ m_X^s)$ vanishes on $Y$. Therefore by $V_Y$ being Euler-like we obtain
\[
\L_{V_Y}(g\circ m_X^s)(v) = (g\circ m_X^s)(v) + r_s(v)
\]
where $r_s$ vanishes to order 2 on $Y$, i.e.
\begin{equation}\label{remainder_vanish}
r_s\lvert_Y=0,\: Dr_s\lvert_Y = 0. 
\end{equation}

Now
\begin{align*}
    (\mathcal{L}_{\widetilde{V_Y}}g)(v) &= \frac{d}{dt}\Big\lvert_0 \left(g\circ m_X^{(f\circ \rho_X)(v)}\circ m_Y^{\exp(t)}\circ m_X^{{(f\circ \rho_X)(v)}^{-1}} \right)(v) \\
    &= \frac{d}{dt}\Big\lvert_0 \left (g\circ m_X^{(f\circ \rho_X)(v)}\circ m_Y^{\exp(t)} \right)\left(m_X^{{(f\circ \rho_X)(v)}^{-1}}(v)\right) \\
    &= \L_{V_Y}\left(g\circ m_X^{(f\circ \rho_X)(v)} \right)\left(m_X^{{(f\circ \rho_X)(v)}^{-1}}(v)\right) \\
    &= \left (g\circ m_X^{(f\circ \rho_X)(v)} \right)\left(m_X^{{(f\circ \rho_X)(v)}^{-1}}(v)\right) + r_{(f\circ \rho_X)(v)}\left(m_X^{(f\circ \rho_X)(v)^{-1}}(v)\right) \\
    &= g(v) + \left(r_{(f\circ \rho_X)(v)}\circ m_X^{{(f\circ \rho_X)(v)}^{-1}}\right)(v)
\end{align*}
so it is enough to show that $v\mapsto \left(r_{(f\circ \rho_X)(v)}\circ m_X^{(f\circ \rho_X)(v)^{-1}}\right)(v)$ vanishes to order $2$ on $Y$. One can check that it vanishes on $Y$ so it remains to show that so does its differential.

Let $y\in Y$. Since (\ref{remainder_vanish}) holds for all $s\in\R_{>0}$, it follows that $\frac{d}{ds}\Big\lvert_{s_0}r_s(y)=0$ for all $s_0\in\R_{>0}$. Thus
\begin{align*}
& D\left[v\mapsto \left(r_{(f\circ \rho_X)(v)}\circ m_X^{(f\circ \rho_X)(v)^{-1}}\right)(v)\right](y) \\
=& \underbrace{Dr_{(f\circ \rho_X)(y)}}_{\equiv 0\text{ on }Y}
\left( \underbrace{m_X^{(f\circ \rho_X)(y)^{-1}}(y)}_{\in Y} \right)\cdot D\left[ v\mapsto m_X^{(f\circ \rho_X)(v)^{-1}}(v)\right](y) + \\ 
& + \underbrace{\frac{d}{ds}\Big\lvert_{(f\circ \rho_X)(y)}r_s(y)}_0\cdot D(f\circ \rho_X)(y) \\
= & 0
\end{align*}
which proves $\widetilde{V_Y}$ is Euler-like along $Y$.
\begin{claim_nn}
$\widetilde{V_Y}$ is a convenient Euler-like vector field along $Y$.
\end{claim_nn}
The fact that $\widetilde{V_Y}$ is complete follows from (\ref{edited_euler_flow}) and the assumption that $V_X,V_Y$ are complete.

To prove \ref{conv_euler_conical} note that
\[
\{\widetilde{V_Y} \neq 0 \} = \{v\in M_d\setminus Y: m_X^{(f\circ \rho_X)(v)}(v)\in T_Y \}
\]
so for all $v$ in the LHS we have
\[
\lim_{t\to-\infty} \widetilde{m}_Y^{\exp(t)}(v) = \left( m_X^{(f\circ \rho_X)(v)} \circ m_Y^0 \circ  m_X^{{(f\circ \rho_X)}^{-1}(v)} \right)(v) \in Y
\]
which implies
\[
\{\widetilde{V_Y} \neq 0 \} \subset \{v\in M_d\setminus Y:\: \lim_{t\to-\infty} \widetilde{m}_Y^{\exp(t)}(v) \in Y \}.
\]
The other direction is always true.

For \ref{conv_euler_extendable}, let $\pi$ be an extension of $m_Y^0$ to some neighbourhood $U$ of $\overline{T_Y}$. This extension exists by the assumption that $V_Y$ is convenient. Then
\[
\widetilde{U} \coloneqq \{v\in M_d: m_X^{(f\circ \rho_X)(v)}(v)\in U\}
\]
contains $\overline{T_Y}$ and we claim it is open. Indeed, let $v_0\in \widetilde{U}$ and let $m_{X,f}:M_d\times \R_{>0}\to M_d$ be the smooth function
\[
m_{X,f}(v, t) = m_X^{f(t)}(v).
\]

Since $m_{X,f}(v_0, \rho_X(v_0))\in U$, there exist a neighbourhood $W_1\subset M_d$ of $v_0$ and an interval $(a,b)$ containing $\rho_X(v_0)$ such that $W_1\times (a,b) \subset m_{X,f}^{-1}(U)$. Denote $W_2 = \rho_X^{-1}(a,b)$ and observe that for any $v\in W_1\cap W_2$ we have
\[
m_X^{(f\circ \rho_X)(v)}(v) = m_{X,f}(v, \rho_X(v)) \in U
\]
so $W_1\cap W_2$ is a neighbourhood of $v_0$ which is contained in $\widetilde{U}$.

Now define $\widetilde{\pi}:\widetilde{U}\to Y$ by
\[
\widetilde{\pi}(v) = \left( m_X^{(f\circ \rho_X)(v)} \circ \pi \circ  m_X^{{(f\circ \rho_X)}^{-1}(v)} \right)(v).
\]
This is an extension of $\widetilde{m}_Y^0$ to $\widetilde{U}$, so it proves \ref{conv_euler_extendable}.
\end{proof}
\begin{lem}\label{conjugating_by_sqrt_vector}
With the notation of \ref{conjugating_by_some_func}, assume in addition that
\[
\forall 0<t\leq 1/2:\: f(t)=\sqrt{t}.
\]

Then for $v\in \widetilde{T_Y}\cap \{\rho_X < 1/2\}$ the following holds 
\[
[V_X,\widetilde{V_Y}](v)=0.
\]
\end{lem}
\begin{proof}
Let $v\in \widetilde{T_Y}\cap \{\rho_X<1/2\}$ and choose $\varepsilon>0$ small enough so that
\[
\forall s\in (-\varepsilon, \varepsilon):\:{\exp(s)}^2\cdot \rho_X(v) < 1/2.
\]

It follows that for $s\in(-\varepsilon, \varepsilon)$ we have
\begin{align*}
& (f\circ \rho_X \circ m_X^{\exp(s)})(v) \\
=& f\left( {\exp(s)}^2\cdot \rho_X(v)\right) \\ 
=& \exp{s} \cdot (f\circ \rho_X)(v)
\end{align*}
and therefore
\begin{align*}
& \left( \widetilde{m}_Y^{\exp(t)} \circ m_X^{\exp(s)}\right)(v) \\ 
=& \left(m_X^{(f\circ \rho_X \circ m_X^{\exp(s)})(v)} \circ m_Y^{\exp(t)} \circ  m_X^{{(f\circ \rho_X \circ m_X^{\exp(s)})(v)}^{-1}} \right)(m_X^{\exp(s)}(v)) \\
=& \left(m_X^{\exp(s)\cdot (f\circ \rho_X)(v)} \circ m_Y^{\exp(t)} \circ  m_X^{{\exp(s)}^{-1}\cdot {(f\circ \rho_X)(v)}^{-1}} \circ m_X^{\exp(s)} \right)(v) \\
=& \left( m_X^{\exp(s)} \circ m_X^{(f\circ \rho_X)(v)} \circ m_Y^{\exp(t)} \circ  m_X^{{(f\circ \rho_X)}^{-1}(v)} \right)(v) \\
= & \left( m_X^{\exp(s)} \circ \widetilde{m}_Y^{\exp(t)} \right) (v)
\end{align*}
which implies $[V_X,\widetilde{V_Y}](v)=0$.
\end{proof}
\begin{lem}\label{conjugating_by_sqrt_distace}
With the notation and assumptions of \ref{conjugating_by_sqrt_vector}, define
\begin{equation}\label{edited_distance}
\widetilde{\rho_Y}(v) := \left( \rho_Y\circ m_X^{{(f\circ \rho_X)}^{-1}(v)} \right)(v).
\end{equation}

Then $\widetilde{\rho_Y}(v)$ is a distance function on $\widetilde{T_Y}$ and for $v\in \widetilde{T_Y}\cap \{\rho_X < 1/2\}$ the following holds
\[
\L_{V_X}\widetilde{\rho_Y} = 0.
\]
\end{lem}
\begin{proof}
One can check that $\widetilde{\rho_Y}$ is a distance function using Equation (\ref{edited_euler_flow}).

Let $v\in \widetilde{T_Y}\cap \{\rho_X<1/2\}$. 
Similarly to the  proof of Lemma \ref{conjugating_by_sqrt_vector}, choose $\varepsilon>0$ small enough such that for all $s\in(-\varepsilon, \varepsilon)$ we have
\[
(f\circ \rho_X \circ m_X^{\exp(s)})(v)=\exp(s) \cdot (f\circ \rho_X)(v).
\]

Now 
\begin{align*}
& \left( \widetilde{\rho_Y} \circ m_X^{\exp(s)} \right)(v) \\
=& \left( \rho_Y\circ m_X^{{(f\circ \rho_X \circ m_X^{\exp(s)})}^{-1}(v)} \right)\left( m_X^{\exp(s)}(v) \right) \\
=& \left( \rho_Y\circ m_X^{{\exp(s)}^{-1}\cdot {(f\circ \rho_X)(v)}^{-1}} \circ m_X^{\exp(s)} \right)(v) \\
=& \left( \rho_Y\circ m_X^{{(f\circ \rho_X)}^{-1}(v)} \right)(v) \\
=& \widetilde{\rho_Y}(v) 
\end{align*}
and the result follows.
\end{proof}
\begin{cor}\label{conjugating_commute}
With the notation and assumptions of \ref{conjugating_by_sqrt_distace}, let $h\coloneqq h_{\nicefrac{1}{4},\nicefrac{1}{2}}$ (Notation \ref{bump_function}) and let $\widetilde{T_X}$ be the shrinking of $T_X$ by $h$.

Then $\widetilde{T_X},\widetilde{T_Y}$ commute.
\end{cor}
\begin{proof}
First we show $\widetilde{T_X},\widetilde{T_Y}$ pre-commute. 
For all $t\in \R$ and $v\in T_X\cap \widetilde{T_Y}$ we have
\begin{align*}
& \left(\rho_X \circ \widetilde{m}_Y^{\exp(t)} \right)(v) \\
=& \left( \rho_X \circ m_X^{(f\circ \rho_X)(v)} \circ m_Y^{\exp(t)} \circ  m_X^{{(f\circ \rho_X)}^{-1}(v)}\right)(v) \\
=& {(f\circ \rho_X)(v)}^2 \cdot \left( \rho_X \circ m_Y^{\exp(t)} \circ  m_X^{{(f\circ \rho_X)}^{-1}(v)}\right)(v) \\
=& {(f\circ \rho_X)(v)}^2 \cdot \left( \rho_X \circ  m_X^{{(f\circ \rho_X)}^{-1}(v)}\right)(v) \\
=& \rho_X(v)
\end{align*}
and 
\begin{align*}
& \left(m_X^0 \circ \widetilde{m}_Y^{\exp(t)} \right)(v) \\
=& \left( m_X^0 \circ m_X^{(f\circ \rho_X)(v)} \circ m_Y^{\exp(t)} \circ  m_X^{{(f\circ \rho_X)}^{-1}(v)}\right)(v) \\
=& \left( m_X^0 \circ m_Y^{\exp(t)} \circ  m_X^{{(f\circ \rho_X)}^{-1}(v)}\right)(v) \\
=& \left( m_X^0 \circ  m_X^{{(f\circ \rho_X)}^{-1}(v)}\right)(v) \\
=& m_X^0(v).
\end{align*}

Now Lemma \ref{shrinking_distance_level_sets} implies that level sets of $(\widetilde{m}_X^0, \widetilde{\rho_X})$ are contained in those of $(m_X^0, \rho_X)$ so $\widetilde{T_X},\widetilde{T_Y}$ pre-commute.

Now let $v\in \widetilde{T_Y}\cap \widetilde{T_X} \subset  \widetilde{T_Y}\cap \{\rho_X < 1/2\}$. By Lemma \ref{conjugating_by_sqrt_vector} we know that
\[
[V_X,\widetilde{V_Y}](v) =0
\]
but $\widetilde{V_X}= (h\circ \rho_X)\cdot V_X$ so
\begin{align*}
& [\widetilde{V_Y}, \widetilde{V_X}](v) \\
=& [\widetilde{V_Y}, (h\circ \rho_X)\cdot V_X](v) \\
=& \underbrace{\L_{\widetilde{V_Y}}{(h\circ \rho_X)}(v)}_0\cdot V_X(v) + (h\circ \rho_X)(v) \cdot \underbrace{[\widetilde{V_Y}, V_X](v)}_0 \\
=& 0.
\end{align*}

Finally, Lemma \ref{conjugating_by_sqrt_distace} implies
\[
\L_{V_X}\widetilde{\rho_Y}(v) = 0
\]
so that
\begin{align*}
& \L_{\widetilde{V_X}}\widetilde{\rho_Y}(v) \\
=& (h\circ \rho_X)(v)\cdot \L_{V_X}\widetilde{\rho_Y}(v) \\
=& 0.
\end{align*}
\end{proof}
\begin{definition}
With the notation and assumptions of \ref{conjugating_by_sqrt_distace}, $\left(\widetilde{V_Y}, \widetilde{T_Y}, \widetilde{m}_Y^t, \widetilde{\rho_Y} \right)$ is the conjugation of $T_Y$ by $T_X$.
\end{definition}
\begin{rem}\label{conjugating_same_outside}
On $M_d\setminus T_X$ the vector field $\widetilde{V_Y}$ and the maps $\widetilde{m}_X^t,\widetilde{\rho_Y}$ are the same as $V_Y$ and $m_Y^t,\rho_Y$.
\end{rem}
\subsection{Proof of Theorem \ref{commutative_control_data_theorem}}
\hfill

We modify the given collection using descending induction on $d$, starting from the highest dimension of a stratum $n_0$.

\underline{Assumption for construction step $d$}:
There exists a collection ${\{\left(V_X, T_X, m_X^t, \rho_X \right)\}}_{X\in \S}$ of convenient tubular neighbourhoods of strata with the following properties:
\begin{enumerate}
    \item $G$-equivariant.
    \item\label{in_proof_2_adjusted} Adjusted with regard to $\S$.
    \item Tangential with regard to $\S$.
    \item Pre-commutative.
    \item For $X,Y\in\S$ with $\dim{X},\dim{Y} > d$, the tubular neighbourhoods $T_X,T_Y$ commute.
\end{enumerate}

\underline{Construction step $n_0$}: The given collection already satisfies the induction assumption for the next step $d=n_0 - 1$ as two strata of the same dimension are not comparable and therefore their tubular neighbourhoods commute.

\underline{Construction step $d$}: For all $Y\in \S$ with $\dim{Y}>d$, let $\widetilde{T_Y}$ be the convenient tubular neighbourhood of $Y$ obtained by conjugating $T_Y$ with regard to all $X<Y$ with $\dim{X}=d$.

Note that for a given $Y$ of dimension $>d$, Property (\ref{in_proof_2_adjusted}) implies that ${\{T_Y\cap T_X\}}_{\dim{X}=d}$ are disjoint so by Remark \ref{conjugating_same_outside} this modification does not depend on the order of these conjugations.

For all $X\in\S$ with $\dim{X}=d$, denote by $\widetilde{T_X}$ the shrinking of $T_X$ by $h\coloneqq h_{\nicefrac{1}{4},\nicefrac{1}{2}}$. For $Y\in S$ with $\dim{Y}<d$ denote $\widetilde{T_Y}=T_Y$. We claim that the collection
\[
{\Big\{\left(\widetilde{V_Y}, \widetilde{T_Y}, \widetilde{m}_Y^t, \widetilde{\rho_Y} \right)\Big\}}_{Y\in \S}
\]
satisfies the induction assumption for step $d-1$.

\underline{$G$-equivariant}: Let $Y\in\S$. If $\dim{Y}<d$ then this is given by the induction assumption. 
If $\dim{Y}=d$ this follows from the fact that shrinking tubular neighbourhoods preserves $G$-equivariance. 
If $\dim{Y}>d$ this follows from Equations (\ref{edited_euler_flow}), (\ref{edited_distance}) and the induction assumption. 

\underline{Adjusted with regard to $\S$}: \ref{ad4} is implies by the fact all the tubular neighbourhoods are convenient and Proposition \ref{distance_close}.

To prove that \ref{ad1},\ref{ad2} and \ref{ad3} are also satisfied one can use the induction assumption, Property \ref{conv_euler_conical} and Equation (\ref{edited_euler_vector}). We will prove \ref{ad1} as an example, the rest are similar.

The interesting part is to show that $\widetilde{T_Z}\cap \widetilde{T_Y}\neq \emptyset$ implies $Z,Y$ are comparable, as the other direction follows from the definition of the order relation.

For $\dim{Z},\dim{Y}<d$ this follows from the induction assumption as the tubular neighbourhoods remain the same. For $\dim{Y}=d,\dim{Z}\leq d$ it follows from $\widetilde{T_Y}\subset T_Y$.

If $Z=X$ with $\dim{X}=d$ and $\dim{Y}>d$, let $v\in \widetilde{T_X}\cap \widetilde{T_Y}$. Then
\begin{align*}
0\neq & \widetilde{V_Y}(v) =
{Dm_X^{(f\circ \rho_X)(v)}}_{m_X^{{(f\circ \rho_X)}^{-1}(v)}(v)} \cdot V_Y\left( m_X^{{(f\circ \rho_X)}^{-1}(v)}(v)\right)
\end{align*}
and therefore
\[
V_Y\left( m_X^{{(f\circ \rho_X)}^{-1}(v)}(v)\right) \neq 0
\]
which implies that $m_X^{{(f\circ \rho_X)}^{-1}(v)}(v) \in T_Y$. But $v\in \widetilde{T_X}\subset T_X$ so $m_X^{{(f\circ \rho_X)}^{-1}(v)}(v) \in T_X$ and the result follows by the induction assumption.

If $\dim{Z},\dim{Y}>d$, let $v\in \widetilde{T_Z}\cap \widetilde{T_Y}$. If there exist no $X\in\S$ with $\dim{X}=d$ such that $v\in T_X$ then $\widetilde{V_Y}(v)=V_Y(v),\widetilde{V_Z}(v)=V_Z(v)$ and the result follows from the induction assumption.

Otherwise, let $X$ be such a stratum. Then similarly to before
\begin{align*}
& V_Y\left( m_X^{{(f\circ \rho_X)}^{-1}(v)}(v)\right) \neq 0 \\
& V_Z\left( m_X^{{(f\circ \rho_X)}^{-1}(v)}(v)\right) \neq 0
\end{align*}
which implies $m_X^{{(f\circ \rho_X)}^{-1}(v)}(v)\in T_Z\cap T_Y$ and the result follows from the induction assumption.

\underline{Tangential with regard to $\S$}: For $Y\in\S$ with $\dim{Y}\leq d$ this follows from the induction assumption and the fact that $\widetilde{V_Y}$ either remains the same or is multiplied by a function.

For $Y\in\S$ with $\dim{Y}>d$ this follows from Equation (\ref{edited_euler_vector}) and the following fact.
For $s>0$ and $X,Z\in\S$ with $\dim{X}=d$ and $X<Y<Z$, the map $m_X^s$ preserves $Z$ and therefore
\[
\text{Im}\left( Dm_X^s\lvert_{TZ} \right) \subset TZ.
\]

\underline{Pre-commutative}: Let $Y<Z$. For $\dim{Z}\leq d$ this follows from the induction assumption.

For $Y=X$ with $\dim{X}=d$ this follows from Corollary \ref{conjugating_commute}.

Assume $d<\dim{Y}<\dim{Z}$ and let $v\in \widetilde{T_Y}\cap \widetilde{T_Z}$. If there exist no $X\in\S$ with $\dim{X}=d$ such that $v\in T_X$, then $\widetilde{m}_Y^0,\widetilde{\rho_Y}$ and $\widetilde{m}_Z^t$ are the same as $m_Y^0,\rho_Y$ and $m_Z^t$ so this follows from the induction assumption. Thus we assume there exist $X\in\S$ with $\dim{X}=d$ such that $v\in T_X$.

Using $\rho_X\circ \widetilde{m}_Z^s=\rho_X$ and the induction assumption we get
\begin{align*}
& \left(\widetilde{m_Y^0} \circ \widetilde{m_Z^s} \right)(v) \\
=& \left( m_X^{(f\circ \rho_X \circ \widetilde{m_Z^s})(v)} \circ m_Y^0 \circ m_X^{{(f\circ \rho_X \circ \widetilde{m_Z^s})}^{-1}(v)} \circ \widetilde{m_Z^s}\right)(v) \\
=& \left( m_X^{(f\circ \rho_X)(v)} \circ m_Y^0 \circ m_X^{{(f\circ \rho_X)}^{-1}(v)} \circ \widetilde{m_Z^s}\right)(v) \\
=& \left( m_X^{(f\circ \rho_X)(v)} \circ m_Y^0 \circ m_X^{{(f\circ \rho_X)}^{-1}(v)} \circ m_X^{(f\circ \rho_X)(v)} \circ m_Z^s \circ m_X^{{(f\circ \rho_X)}^{-1}(v)}\right)(v) \\
=& \left( m_X^{(f\circ \rho_X)(v)} \circ m_Y^0 \circ m_X^{{(f\circ \rho_X)}^{-1}(v)}\right)(v) \\
=& \widetilde{m_Y^0}(v)
\end{align*}
and similarly for $\widetilde{\rho_Y}\circ \widetilde{m}_Z^s$.

\underline{Strata of dimension $>d-1$ commute}: Let $X<Y$ with $\dim{X}=d$. Then $\widetilde{T_X},\widetilde{T_Y}$ commute by Corollary \ref{conjugating_commute}. 

Let $Y<Z$ with both dimensions $>d$ and let $v\in \widetilde{T_Y}\cap \widetilde{T_Z}$. Again, the interesting case is when there exist $X\in\S$ with $\dim{X}=d$ such that $v\in T_X$.

With the same arguments as in the proof of pre-commutativity, now using the assumption that $T_Y,T_Z$ commute we obtain
\begin{align*}
& \left(\widetilde{m_Y^t} \circ \widetilde{m_Z^s} \right)(v) \\
=& \left( m_X^{(f\circ \rho_X \circ \widetilde{m_Z^s})(v)} \circ m_Y^t \circ m_X^{{(f\circ \rho_X \circ \widetilde{m_Z^s})}^{-1}(v)} \circ \widetilde{m_Z^s}\right)(v) \\
=& \left( m_X^{(f\circ \rho_X)(v)} \circ m_Y^t \circ m_X^{{(f\circ \rho_X)}^{-1}(v)} \circ m_X^{(f\circ \rho_X)(v)} \circ m_Z^s \circ m_X^{{(f\circ \rho_X)}^{-1}(v)}\right)(v) \\
=& \left( m_X^{(f\circ \rho_X)(v)} \circ m_Y^t \circ m_Z^s \circ m_X^{{(f\circ \rho_X)}^{-1}(v)}\right)(v) \\
=& \left( m_X^{(f\circ \rho_X)(v)} \circ m_Z^s \circ m_Y^t \circ m_X^{{(f\circ \rho_X)}^{-1}(v)}\right)(v) \\
=& \left( m_X^{(f\circ \rho_X)(v)} \circ m_Z^s \circ m_X^{{(f\circ \rho_X)}^{-1}(v)} \circ m_X^{(f\circ \rho_X)(v)} \circ m_Y^t \circ m_X^{{(f\circ \rho_X)}^{-1}(v)}\right)(v) \\
=& \left( m_X^{(f\circ \rho_X \circ \widetilde{m_Y^t})(v)} \circ m_Z^s \circ m_X^{{(f\circ \rho_X \circ \widetilde{m_Y^t})}^{-1}(v)} \circ \widetilde{m_Y^y}\right)(v) \\
=& \left(\widetilde{m_Z^s} \circ \widetilde{m_Y^t} \right)(v)
\end{align*}
and similarly for $\widetilde{\rho_Z}\circ \widetilde{m}_Y^t$.

This finishes the construction of step $d$ and therefore the proof. \qed

\section{Neighbourhood smooth weak deformation retraction}\label{weak_deformation_section}
\begin{thm}\label{weak_deformation_theorem}
Let $M$ be a manifold and let $(C,\S)$ be a stratified subset of it. Assume there exist commutative tangential control data ${\{ \left(T_X, \rho_X, m_X^t, V_X \right)\}}_{X\in\S}$.

Then there exists a smooth weak deformation retraction from the open set
\[
\bigcup_{X\in\S}\{\rho_X<1\}
\]
to $C$, which is the restriction of a smooth homotopy $F:M\times [0,1]\to M$.

If $G$ is a compact Lie group acting on $M$, the stratification $\S$ is a $G$-stratification and the control data is $G$-equivariant then the smooth homotopy $F$ can be constructed to be $G$-equivariant.
\end{thm}

By \ref{ad3}, the tubular neighbourhoods $T_X, T_Y$ are disjoint for $X,Y\in\S$ of the same dimension. Denote 
\begin{align*}
C_d &= \bigsqcup_{\dim{X}=d}X \\
T_d &= \bigsqcup_{\dim{X}=d}T_X \subset M_d
\end{align*}
and let $\rho_d,m_d^t$ be the functions defined on $T_d$ such that for each $X\in\S$ with $\dim{X}=d$ we have
\[
\rho_d\lvert_{T_X} = \rho_X,\:m_d^t\lvert_{T_X} = m_X^t.
\]

Recall Notation \ref{bump_function}.
\begin{prop}
Define $\varphi_d:M_d\to [0,1]$ by
\[
\varphi_d(v) = (h_{2,3}\circ \rho_d)(v) \cdot \prod_{0 \leq i <d}\left(1 - (h_{1,2}\circ \rho_i)(v) \right).
\]

Then $\varphi_d$ can be smoothly extended to $M$ by setting it to $0$ on $M\setminus M_d = C_{<d}$.
\end{prop}
\begin{proof}
First, by \ref{ad4} and Corollary \ref{extending_functions} the above extension is well defined and smooth on $M_d$. By definition it is $\equiv 0$ on
\[
\bigcup_{i<d}\{\rho_i < 1\}
\]
which is a neighbourhood of $C_{<d}$.
\end{proof}
\begin{prop}\label{def_of_fd}
Define $f_d:[0,1]\times M\to M$ by
\[
f_d(t,x) = f_d^t(x) \coloneqq 
\begin{cases}
m_d^{1-t\varphi_d(x)}(x) & x\in M_d \\
x & x\in C_{<d}
\end{cases}
\]

Then $f_d$ is smooth.
\end{prop}
\begin{proof}
$f_d$ is smooth on $[0,1]\times M_d$ by the definition of $\varphi_d$. One can check that it is the projection to $M$ on
\[
[0,1] \times \bigcup_{i<d}\{\rho_i < 1\}
\]
which is a neighbourhood of $[0,1] \times C_{<d}$.
\end{proof}

From now on, denote
\[
\{ \rho_i \geq a\} = M_i \setminus \{ \rho_i < a\}.
\]
It is a closed set of $M_i$.
\begin{lem}\label{fd_propeties}
Let $f_d:[0,1]\times M\to M$ be defined as in Proposition \ref{def_of_fd}. It has the following properties:
\begin{enumerate}
    \item $f_d^0 = \operatorname{Id}$.
    \item For all $t\in[0,1]$, the map $f_d^t$ is the identity on 
    \[
    \{\rho_d \geq 3\} \cup \left(\bigcup_{i<d}\{\rho_i \leq 1 \} \right).
    \]
    \item For all $t\in[0,1]$, the map $f_d^t$ preserves $C_{\geq d} = C\cap M_d$.
    \item\label{fd_preserves_distances} Along each curve $t\mapsto f_d^t(x)$, the function $\rho_d$ is weakly decreasing and the functions $\rho_i$ for $i\neq d$ are constant.
    \item\label{fd_retracts} $f_d^1$ takes the set
    \[
    \{\rho_d \leq 2\} \cap \left(\bigcap_{i<d}\{\rho_i \geq 2\} \right)
    \]
    to $C_d$.
\end{enumerate}

If $G$ is a compact Lie group acting on $M$, the stratification $\S$ is a $G$-stratification and the control data is $G$-equivariant then $f_d^t$ is $G$-equivariant.
\end{lem}
\begin{proof}\hfill
\begin{enumerate}
    \item On $M_d$, the map $f_d^0=m_d^1$ is the identity. It is also the identity on $C_{<d}$ by definition.
    \item $\varphi_d$ is $0$ on this set and therefore $f_d^0 = m_d^1$ is the identity.
    \item Follows from the assumption that the collection is tangential and $\text{Im}\left(m_d^0\right)\subset C_d$.
    \item Follows from commutativity.
    \item $\varphi_d\equiv 1$ on this set, and so $f_d^1$ restricted to this set is $m_d^0$.
\end{enumerate}

The equivariant part follows from the definition of $f_d$.
\end{proof}
Let 
\[
n_0 \coloneqq \max_{X\in\S}{\{\dim{X}\}}.
\]
\begin{cor}\label{fd_preserves_nbhd}
The map $f_d^t$ takes
\[
\bigcup_{1\leq i \leq n_0}\{\rho_i<1\}
\]
to itself.
\end{cor}
\begin{proof}
Follows from Property (\ref{fd_preserves_distances}).
\end{proof}
\begin{proof}[Proof of Theorem \ref{weak_deformation_theorem}]
We proceed by descending induction on $d$.

\underline{Induction assumption for step $d$}: There exists a smooth map
\[
F_{d+1}:[0,1]\times M\to M
\]
inducing a family of smooth maps $F_{d+1}^t:M\to M$ such that:
\begin{enumerate}
    \item The map $F_{d+1}^t$ is $G$-equivariant.
    \item For all $t\in[0,1]$, the map $F_{d+1}^t$ is the identity on $\bigcup_{i<d+1}\{\rho_i<1\}$.
    \item For all $t\in[0,1]$, the map $F_{d+1}^t$ takes $C$ to itself.
    \item Along each curve $t\mapsto F_{d+1}^t(x)$, the functions $\rho_i$ for $i< d+1$ are constant.
    \item For all $t\in[0,1]$, the map $F_{d+1}^t$ takes \[
    \bigcup_{1\leq i \leq n_0}\{\rho_i<1\}
    \]
    to itself.
    \item\label{retract_to_cd} The map $F_{d+1}^1$ takes
    \[
    \left( \bigcap_{i<d+1}\{\rho_i \geq 2\}\right) \cap \left( \bigcup_{i\geq d+1}\{\rho_i \leq 1\} \right)
    \]
    to $C_{\geq d+1}$
\end{enumerate}

\underline{Induction step $n_0$}: Take $F_{n_0}^t = f_{n_0}^t$. Lemma \ref{fd_propeties} and Corollary \ref{fd_preserves_nbhd} imply that it satisfies the induction assumption for the next step.

\underline{Induction step $d$ --- construction of $F_d^t$}:
Let $h':[0,1]\rightarrow[0,1]$ be a smooth monotone function that is $\equiv0$ in a neighbourhood of $0$ and $\equiv 1$ in a neighbourhood of $1$. Define
\[
F_d^t := \begin{cases}
F_{d+1}^{h'(2t)} & t\leq 1/2 \\
f_d^{h'(2t-1)}\circ F_{d+1}^1 & t\geq 1/2
\end{cases}
\]
which is smooth. We claim it satisfies the induction assumption for the next step. 

Indeed, the first 5 properties follow by the induction assumption, Lemma \ref{fd_propeties} and Corollary \ref{fd_preserves_nbhd}.
For Property (\ref{retract_to_cd}), let
\[
x\in \left( \bigcap_{i<d}\{\rho_i \geq 2\}\right) \cap \left( \bigcup_{i\geq d}\{\rho_i \leq 1\} \right)
\]
and split to the following cases:

\textbf{Case $\rho_d(x)\geq 2$:} Then $x$ is in the set
\[
\left( \bigcap_{i<d+1}\{\rho_i \geq 2\}\right) \cap \left( \bigcup_{i\geq d+1}\{\rho_i \leq 1\} \right)
\]
so $F_{d+1}^1(x)\in C_{\geq d+1}$. Since $f^1_d$ preserves $C_{\geq d}$ we get that 
\[
F_d^1(x)=(f_d^1\circ F_{d+1}^1)(x)\in C_{\geq d}.
\]

\textbf{Case $\rho_d(x)\leq 2$:} By assumption
\[
(\rho_d\circ F_{d+1}^1)(x) = \rho_d(x) \leq 2
\]
and for all $i<d$:
\[
(\rho_i\circ F_{d+1}^1)(x) = \rho_i(x) \geq 2.
\]

Therefore
\[
F_{d+1}^1(x) \in \{\rho_d \leq 2\} \cap \left( \bigcap_{i<d}\{\rho_i \geq 2\} \right)
\]
so by Property (\ref{fd_retracts}) of $f_d$ we get
\[
F_d^1(x) = (f_d^1\circ F_{d+1}^1)(x)\in C_{\geq d}
\]
which proves the inductive construction. 

Finally, $F=F_0^t$ is the required smooth homotopy. Restricted to
\[
\bigcup_{1\leq i\leq n_0}\{\rho_i<1\}
\]
it is a $G$-equivariant smooth weak deformation retraction to $C$.
\end{proof}
\begin{rem}\label{stratified_subspace_deformation}
With the assumptions of Theorem \ref{weak_deformation_theorem}, let $(C',\S')$ be a stratified subspace (Definition \ref{stratified_subspace}) of $(C, \S)$.
It is also a stratified subset of $M$, and one readily checks that the collection of tubular neighbourhoods ${\{ \left(T_X, \rho_X, m_X^t, V_X \right)\}}_{X\in\S'}$ satisfies the same properties that the original one does. 

Applying Theorem \ref{weak_deformation_theorem} for $C'\subset M$
and restricting the homotopy to $C$ yields a neighbourhood of $C'$ in $C$  admitting a smooth weak deformation retraction to $C'$, which is $G$-equivariant in the $G$-equivariant case. Here by smooth on $C$ we use the smooth structure as an embedded subspace of $M$ as in \cite[Section~1.3]{pflaum_book}.

It follows that Theorem \ref{weak_deformation_theorem} can be seen as a smooth version of \cite[Theorem~1.1]{pflaum2017equivariant}. Note that, to achieve smooth maps, we make stronger assumptions and we obtain a weak deformation retraction instead of a strong one.
\end{rem}

\section{Stratification by orbit types}\label{orbit_type_section}

Let $G$ be a compact Lie group, not necessarily connected, acting on a manifold $M$. For $x\in M$ denote by $G_x\subset G$ its stabilizer, and write $H\sim K$ when $H,K$ are conjugate in $G$. Define
\[
M_{(H)} = \{x\in M:\: G_x\sim H\}.
\]

Note that connected components of $M_{(H)}$ are manifolds but may have different dimensions for the same $H\subset G$. This implies that this decomposition of $M$ is too coarse to be a stratification.
\begin{definition}\label{orbit_type_con_comp}
The stratification of $M$ by connected components of orbit types is:
\[
\S_{\text{con}} = \{X:\: X\text{ is a connected component of }M_{(H)}\text{ for }H\subset G\}
\]
\end{definition}
This decomposition is a stratification of $M$ (see \cite{duistermaatkolk} for example) but it is too fine for our purposes --- the strata are not necessarily $G$-invariant in the case $G$ is not connected.

We now present two different refinements of $\{M_{(H)}\}$ that are $G$-stratifications of $M$. We begin with the finer one:
\begin{definition}\label{stratification_by_orbit_type_def}
Let $\pi:M\to M/G$ be the quotient map. Define
\[
\S_G = \{\pi^{-1}(\pi(X)):\: X\in \S_{\text{con}}\}
\]
or equivalently
\[
\S_G = \{G\cdot X:\: X\in \S_{\text{con}}\}
\]
\end{definition}
\begin{prop}
$\S_G$ is a $G$-stratification of $M$.
\end{prop}
\begin{proof}
Local finiteness and fixed dimension of strata follow from the same properties of $\S_{\text{con}}$.

To prove the condition of the frontier, first note that for each closed set $A\subset M$, the set $G\cdot A$ is closed. Now let $G\cdot X, G\cdot Y\in \S_G$ where $X,Y\in \S_{\text{con}}$ and assume $G\cdot X \cap \overline{G\cdot Y}\neq \emptyset$. We may assume $X\cap \overline{Y} \neq \emptyset$.
By the condition of the frontier of $\S_{\text{con}}$ we have $X\subset \overline{Y}$ which implies
\[
G\cdot X \subset G \cdot \overline{Y} = \overline{G \cdot Y}
\]
as required.
\end{proof}

The second refinement of $\{M_{(H)}\}$ is the stratification by \textit{normal orbit types}. We will not give a full definition, but the main idea is to take into consideration the \textit{normal representation} of a point --- the representation of the stabilizer of the point on the normal slice of it. Two points have the same \textit{normal orbit type} if their stabilizers are conjugate and these representations are isomorphic in a suitable sense, see \cite[Section~1]{strat_normal_orbit_type} for more details. 

Decomposing $M$ to the different normal orbit types induces a $G$-stratification of $M$, which is in general coarser then $\S_G$ i.e. all points in a given $X\in\S_G$ have the same normal orbit type.

For both stratifications we have:
\begin{thm}[\cite{strat_by_orbit_types_2},{\cite[Theorem~3.5.4]{wall_2016}}]
The $G$-stratification by orbit types is smoothly locally trivial.
\end{thm}
\begin{rem}
Looking at the proofs given in \cite{strat_by_orbit_types_2} and \cite[Theorem~3.5.4]{wall_2016} one can see that the fibers are conical --- this follows from the local model \cite[Theorem~3.3.5]{wall_2016} and the fact that strata of the orbit type stratification of a vector space are conical. 
\end{rem}

\subsection{Critical set of the norm square of the momentum map}\label{crit_set_subsection}
\hfill

Let $\mu:M\to \g^*$ be a momentum map for the action of a compact Lie group $G$ on a symplectic manifold $M$. Fix an $\text{Ad}$-invariant inner product on $\g$ and fix the induced ${\text{Ad}}^*$-invariant inner product on $\g^*$. 

Arms, Marsden and Moncrief \cite{cone_like_no_local_nromal_form} gave a proof for a local structure of $p\in\mu^{-1}(0)$ as a product of a disk and a cone. Holm and Karshon \cite[Section~7]{morse_bott_kirwan_is_local} gave a similar result for $\text{Crit}\norm{\mu}^2$ using the local normal form by Guillemin--Sternberg \cite{guillemin1984normal} and Marle \cite{marle1985modele}. 

Given these results, it is natural to ask whether intersecting these sets with strata of the orbit type stratification $\S_G$ induces a $G$-stratification that is smoothly locally trivial with conical fibers. This is indeed the case:

\begin{thm}\label{crit_set_stratification_thm}
The decomposition
\[
S\coloneqq\{\text{Crit}\norm{\mu}^2\cap X:\:X\in \S_G\}
\]
is a $G$-stratification of $\text{Crit}\norm{\mu}^2$, that is smoothly locally trivial with conical fibers.
\end{thm}
\begin{proof}
Denote $C=\text{Crit}\norm{\mu}^2$.

First, local finiteness and the condition of the frontier follow from these same properties of $\S_G$. 

If we prove the existence of conical neighbourhoods in the sense of Definition \ref{conical_neighbourhood}, then $C\cap X$ is a manifold for $X\in\S_G$ and therefore this decomposition is a stratification of $C$. 
That it is a $G$-stratification is clear, and having a conical neighbourhood is exactly smooth local triviality with conical fibers.

So it is left to show that each point $p\in C$ has a conical neighbourhood with regard to its stratum.
Since this is a local property, we can use the local normal form to reduce to the \textit{Hamiltonian $G$ model} case. 

With the notation of \cite[Section~7]{morse_bott_kirwan_is_local}, we assume that $M$ is an open neighbourhood of the zero section in
\[
G\times_{H}\left(V \times {(\g_\beta/\h)}^* \right)
\]
and our base point is $[1,0,0]$. Here $H\subset G$ is a closed subgroup, $\h$ its Lie algebra, $V$ is a symplectic vector space on which $H$ acts linearly and symplectically, $\beta = \mu([1,0,0])$, the subgroup $G_\beta\subset G$ is the stabilizer of $\beta$ in the coadjoint action of $G$ on $\g^*$ and $\g_\beta$ is its Lie algebra.

Denote by $Q:V\to \h^*$ the quadratic momentum map for the $H$ action on $V$. The momentum map $\mu:M \to \g^*$ is of the form
\begin{equation}\label{moment_map_local_structure}
\mu([g,z,v]) = \text{Ad}^*(g)\left(\beta + Q(z) + v \right)
\end{equation}
where $\h^*$ and ${(\g_\beta/\h)}^*$ are embedded as sub Lie algebras of $\g^*$ using the inner product. 

Let $\widehat{\beta}$ be the image of $\beta$ under the isomorphism $\g^*\to\g$ given by the inner product. By \cite[Lemma~7.3]{morse_bott_kirwan_is_local}, the assumption $[1,0,0]\in C$ implies $\beta\in \h^*$ and therefore $\widehat{\beta}\in \h$. Let $T_\beta$ be the closure in $H$ of the one parameter subgroup generated by $\widehat{\beta}$.

Following the proofs of \cite[Lemmas~7.11 and 7.10]{morse_bott_kirwan_is_local} and using the identification
\[
G\times_{H}\left(V \times {(\g_\beta/\h)}^* \right) \simeq G\times_{G_\beta}\left(G_\beta \times_{H} \left(V \times {(\g_\beta/\h)}^* \right)\right)
\]
one obtains a neighbourhood $U\subset M$ of $[1,0,0]$ such that
\[
C\cap U = \Bigl( G\times_{G_\beta}\left(G_\beta \times_{H}\left(C_V \times \{0\} \right) \right) \Bigr) \cap U
\]
where $C_V \subset V$ is given by
\[
C_V = Q^{-1}(0)\cap V^{T_\beta}
\]
and $V^{T_\beta}$ is the space of vectors fixed by $T_\beta$. Note that 
\[
G\times_{G_\beta}\left(G_\beta \times_{H}\left(C_V \times \{0\} \right) \right) \simeq G\times_{H}\left(C_V \times \{0\} \right)
\]
and therefore
\begin{equation}\label{crit_set_local_structure}
C\cap U \simeq \Bigl( G\times_{H}\left(C_V \times \{0\} \right) \Bigr) \cap U.
\end{equation}

With this local model we can now work similarly to the the proof of \cite[Theorem~2.1]{sjamaar1991stratified}.
A neighbourhood of $[1,0,0]$ is diffeomorphic to a neighbourhood $U'$ of $([1],0,0)$ in
\[
G/H \times V \times {(\g_\beta/\h)}^*
\]
such that the intersection of $C$ with this neighbourhood is identified with
\[
\Bigl( G/H \times C_V \times \{0\} \Bigr) \cap U'
\]
and the intersections of $X\in\S_G$ with this neighbourhood are of the form
\[
\Bigl( G/H \times V_{(K)} \times {(\g_\beta/\h)}^*_{(K)} \Bigr) \cap U'
\]
for closed subgroups $K\subset H$.

Denote by $V^H\subset V$ the symplectic subspace of vectors fixed by $H$ and by $W\subset V$ its symplectic perpendicular, which is $H$-invariant and symplectic. Therefore $V=V^H\times W$ both as a symplectic space and as an $H$ module. For all closed subgroups $K\subset H$ we have
\[
V_{(K)} = V^H \times W_{(K)}.
\]
and in particular $V_{(H)}=V^H\times \{0\}$. Note that $Q_W \coloneqq Q\lvert_W$ is a momentum map for the action of $H$ on $W$ and that
\[
Q^{-1}(0) = V^H \times Q_W^{-1}(0).
\]

Using the above, $V^H\subset Q^{-1}(0)$ and $V^H\subset V^{T_\beta}$, we obtain
\begin{align*}
M_{(H)} \cap C \cap U' &= \Bigl( G/H \times \left(V_{(H)}\cap C_V\right) \times \{0\} \Bigr) \cap U' \\
&= \Bigl( G/H \times V^H \times \{0\} \times \{0\} \Bigr) \cap U'
\end{align*}
and for $K\subset H$ a closed subgroup we have
\begin{equation}\label{crit_set_strat_local_structure}
\begin{aligned}
M_{(K)} \cap C \cap U' &= \Bigl( G/H \times \left(V_{(K)}\cap C_V\right) \times \{0\} \Bigr) \cap U' \\
&= \Bigl( G/H \times V^H \times \left(W_{(K)}\cap C_V \right) \times \{0\} \Bigr) \cap U' \\
&= \Bigl( G/H \times V^H \times \left(W_{(K)}\cap Q_W^{-1}(0) \cap V^{T_\beta} \right) \times \{0\} \Bigr) \cap U'.
\end{aligned}
\end{equation}
Now $Q_W$ is homogeneously quadratic and thus $\left(W_{(K)}\cap Q_W^{-1}(0) \cap V^{T_\beta} \right)$ is invariant under multiplication by $t\in (0,1)$. Taking a neighbourhood of $[1]$ in $G/H$ diffeomorphic to a disk produces a conical neighbourhood of $[1,0,0]$ in $M$.
\end{proof}
\begin{rem}
Let $p\in \text{Crit}\norm{\mu}^2$ with stabilizer subgroup $H$ and $\mu(p)=\beta$. The fact that $\beta$ is preserved by the $H$ action together with Equation (\ref{crit_set_strat_local_structure}) implies that around $p$, the orbit types $M_{(K)}, K\subset H$ which intersect non-trivially with $\text{Crit}\norm{\mu}^2$ must satisfy
\[
T_\beta \subset K.
\]
\end{rem}
\begin{rem}
Using Theorems \ref{crit_set_stratification_thm}, \ref{tangential_control_data_thm}, \ref{commutative_control_data_theorem} and \ref{weak_deformation_theorem} we obtain a neighbourhood of $\text{Crit}\norm{\mu}^2$ that admits a smooth equivariant weak deformation retraction to $\text{Crit}\norm{\mu}^2$. This confirms a conjecture by Harada and Karshon \cite[Remark~4.23]{harada_karshon} and implies that their localization theorem for the Duistermaat--Heckmann distribution \cite[Theorem~4.24]{harada_karshon} applies to $Z=\text{Crit}\norm{\mu}^2$ without additional assumptions.
\end{rem}
Examining the local structure given in Equations (\ref{moment_map_local_structure}) and (\ref{crit_set_local_structure}) one readily sees that there exists a neighbourhood of $\mu^{-1}(0)$ in $M$ such that the intersection of $\text{Crit}\norm{\mu}^2$ with this neighbourhood is $\mu^{-1}(0)$. It follows that $\mu^{-1}(0)$ is a union of connected components of $\text{Crit}\norm{\mu}^2$ and therefore:
\begin{thm}\label{zero_level_set_stratification_thm}
The decomposition
\[
S\coloneqq\{\mu^{-1}(0)\cap X:\:X\in \S_G\}
\]
is a $G$-stratification of $\mu^{-1}(0)$, which is smoothly locally trivial with conical fibers. 
\end{thm}
\begin{rem}
Using Theorems \ref{zero_level_set_stratification_thm}, \ref{tangential_control_data_thm}, \ref{commutative_control_data_theorem} and \ref{weak_deformation_theorem} we obtain a smooth equivariant homotopy inverse to the inclusion map $i:\mu^{-1}(0)\hookrightarrow U$ for some neighbourhood $U$ of $\mu^{-1}(0)$. This confirms a conjecture by Holm and Karshon \cite[Remark~8.2]{morse_bott_kirwan_is_local}.
\end{rem}

\section{Extending the theory to reduced spaces}\label{extension_section}
The aim of this section is to obtain results similar to Theorems \ref{tangential_control_data_thm}, \ref{commutative_control_data_theorem}, \ref{weak_deformation_theorem} and \ref{zero_level_set_stratification_thm} that apply to reduced spaces. We begin by discussing a notion of smooth structures on stratified spaces that are not given as stratified subsets of manifolds. There are a couple of relevant structures in the literature:
\begin{enumerate}
    \item Stratified subcartesian spaces, see Lusala and Śniatycki \cite{lusala_śniatycki_2011}.
    \item Stratified differentiable spaces, see Crainic and Mestre \cite[Definition~5.1]{orbi_spaces_groupoids}.
    \item Stratified spaces with a smooth structure, see Pflaum \cite[Section~1.3]{pflaum_book}.
\end{enumerate}
Our motivating examples, orbit spaces $M/G$ and reduced spaces $\mu^{-1}(0)/G$, provide examples of each of these structures. We will work with the third structure, as some relevant literature on reduced spaces is already written using this language. 
\begin{definition}[{\cite[Section~1.3]{pflaum_book}}]\label{smooth_structure}
Let $A$ be a topological space with a fixed stratification $\S$. A \textbf{singular chart} of $A$ consists of an open set $U\subset A$, an open set $\Omega\subset \R^n$ and a map $\theta:U\to \Omega$ satisfying:
\begin{enumerate}
    \item $\theta(U)$ closed in $\Omega$.
    \item $\theta$ is a homeomorphism onto its image.
    \item For all $X\in\S$, the map $\theta\lvert_X$ is a diffeomorphism with an embedded submanifold of $\Omega$.
\end{enumerate}

Two singular charts
\begin{align*}
\theta_1&:U_1\to \Omega_1\subset \R^{n_1}\\ 
\theta_2&:U_2\to \Omega_2\subset \R^{n_2}
\end{align*}
are \textbf{compatible} if for all $x\in U_1\cap U_2$ there exist a neighbourhood $U_x\subset U_1\cap U_2$ of $x$, an integer $N\geq \max{\{n_1,n_2\}}$, open neighbourhoods $O_1,O_2\subset \R^N$ of $\theta_1(U_1)\times \{0\}, \theta_2(U_2)\times \{0\}$ respectively and a diffeomorphism $H:O_1\to O_2$ such that
\[
i_{n_1}^N\circ \theta_1 = H\circ i_{n_1}^N\circ \theta_1.
\]
Here $i_{n}^N$ is the natural embedding $R^n\hookrightarrow \R^N$.

A \textbf{singular atlas} $\mathcal{U}$ on $A$ is a collection of pairwise compatible singular charts which cover $A$. Two such atlases are compatible if their union is also a singular atlas, and this defines an equivalence relation. A \textbf{maximal singular atlas} is the atlas containing containing all other compatible atlases as subsets.

A \textbf{smooth structure} on $A$ is a maximal singular atlas.
\end{definition}
\begin{definition}\label{smooth_maps}
Let $A$ be a stratified space with a smooth structure as in Definition \ref{smooth_structure} and let $W$ be an open set. A function $f:W\to \R$ is a \textbf{smooth function} if for every chart $(\theta, U, \Omega)$ with $U\subset W$, the function $f\circ \theta^{-1}:\theta(U)\to \R$ is the restriction of some smooth function on $\Omega$. This defines a sheaf $C_A^\infty(\cdot)$ of smooth functions on $A$.

A map $\varphi:A\to B$ between stratified spaces with smooth structures is a \textbf{smooth map} if
\[
\varphi^*C_B^\infty(\cdot) \subset C_A^\infty(\cdot).
\]
\end{definition}
\begin{example}
A manifold $M$ can be viewed as a stratified space with one stratum. Its (regular) charts can then be viewed as singular charts to obtain the smooth structure.
\end{example}
\begin{example}
A stratified subset of a manifold $i:C\hookrightarrow M$ together with the smooth structure induced from charts of $M$ is a stratified space with a smooth structure. Its sheaf of smooth functions is given by
$i^*C_M^\infty(\cdot)$
and can be identified with the sheaf
\[
C_M^\infty(\cdot) / I_C(\cdot)
\]
where $I_C(\cdot)$ denotes the ideal sheaf of smooth functions vanishing on $C$.
\end{example}
\begin{example}[{\cite[Section~4.4]{pflaum_book}}]\label{linear_space_example}
Let $V$ be a vector space, $H$ a compact Lie group acting on $V$ and $\pi:V\to V/H$ the quotient map. Stratify $V$ by orbit types as described in Definition \ref{stratification_by_orbit_type_def} and stratify $V/H$ by pushing forward the stratification of $V$ using the quotient map. That this induces a stratification of $V/H$ is proved in \cite[Section~4.3]{pflaum_book}.

Let $P(V)^H$ be the algebra of $H$-invariant polynomials on $V$, which is finitely generated by a classical result of Hilbert. Let $\left(\sigma_1,\ldots,\sigma_k \right)$ be a minimal generating set of homogeneous polynomials and define the Hilbert map $\sigma:V\to\R^k$ as
\[
\sigma(v) = \left(\sigma_1(v),\ldots,\sigma_k(v) \right)
\]
and $\overline{\sigma}$ the induced map on $V/H$.

Schwarz \cite{SCHWARZ197563} showed that $\overline{\sigma}:V/H\hookrightarrow \R^k$ is continuous, injective, proper and that
\[
\sigma^*:C_{\R^k}^\infty(\cdot)\to {C_V^\infty(\cdot)}^H
\]
is surjective, where ${C_V^\infty(\cdot)}^H$ is the sheaf of $H$-invariant functions. Taking $\overline{\sigma}:V/H\hookrightarrow \R^k$ to be a global singular chart we obtain a smooth structure on $V/H$ with the sheaf of smooth functions defined as
\[
C_{V/H}^\infty(U) = {C_{V}^\infty\left(\pi^{-1}(U)\right)}^H.
\]
\end{example}
\begin{example}[{\cite[Section~6]{sjamaar1991stratified},\cite[Section~6]{pflaum2001smooth}}]\label{reduced_linear_space_example}
with the notation of Example \ref{linear_space_example}, assume in addition that $V$ is symplectic and that $H$ acts symplectically on $V$ with a momentum map $\mu_V:V\to \h^*$. Let $V_0\coloneqq \mu_V^{-1}(0)/H \subset V/H$ be the reduced space and stratify it by pushing forward the stratification of $\mu_V^{-1}(0)$ described in Theorem \ref{zero_level_set_stratification_thm}. That this induces a stratification of $V_0$ is proved in \cite{sjamaar1991stratified}.

Denote by $I(\cdot)\subset C_V^\infty(\cdot)$ the ideal sheaf of smooth functions vanishing on $\mu_V^{-1}(0)$. Using the fact that $V_0\subset V/G$ is closed and the result of Schwarz we see that
\[
\overline{\sigma}\lvert_{V_0}:V_0\hookrightarrow \R^k
\]
is continuous, injective, proper and
\[
{\sigma\lvert_{\mu^{-1}(0)}}^*:C_{\R^k}^\infty(\cdot)\to {C_V^\infty(\cdot)}^H / I(\cdot)^H
\]
is surjective. Similarly to Example \ref{linear_space_example}, taking $\overline{\sigma}\lvert_{V_0}$ to be a global singular chart we obtain a smooth structure on $V_0$ with the sheaf of smooth functions defined as
\[
C_{V_0}^\infty(U) = {C_{V}^\infty\left(\pi^{-1}(U)\right)}^H / I(\pi^{-1}(U))^H.
\] 
Moreover, identifying $V_0$ with $\overline{\sigma}(V_0)$ gives it the structure of a Whitney (B) stratified subset of $\R^k$.
\end{example}
\begin{example}[{\cite[Section~6]{sjamaar1991stratified},\cite[Section~6]{pflaum2001smooth}}]\label{reduced_manifold_example}
Let $G$ be a compact Lie group acting symplectically on a symplectic manifold $M$ with a momentum map $\mu:M\to \g^*$. Let $\pi:M\to M/G$ be the quotient map and let $M_0 = \mu^{-1}(0)/G\subset M/G$ be the reduced space, stratified by pushing forward the stratification of $\mu^{-1}(0)$ described in Theorem \ref{zero_level_set_stratification_thm}. That this induces a stratification of $M_0$ is proved in \cite{sjamaar1991stratified}. 

Let $x\in M_0$, let $p\in \mu^{-1}(0)$ be a representative of $x$ and let $H$ be the stabilizer subgroup of $p$. Denote by $V$ the symplectic slice for $p\in G\cdot p$ with $\mu_V:V\to \h^*$ a momentum map for the induced $H$ action, and use the notation $V_0,\sigma_V,C_{V_0}^\infty(\cdot)$ of Example \ref{reduced_linear_space_example}.

Denote by $I(\cdot)\subset C_M^\infty(\cdot)$ the ideal sheaf of smooth functions vanishing on $\mu^{-1}(0)$ and define a sheaf $C_{M_0}^\infty(\cdot)$ on $M_0$ by
\[
C_{M_0}^\infty(U) = {C_M^\infty\left(\pi^{-1}(U)\right)}^G / {I\left(\pi^{-1}(U)\right)}^G.
\]

Sjamaar and Lerman \cite[Theorem~5.1]{sjamaar1991stratified} showed that there exist a neighbourhood $U_1\subset M_0$ of $x$, a neighbourhood $U_2\subset V_0$ of $\bar{0}$ and a homeomorphism
\[
\varphi:U_1\to U_2.
\]
such that 
\[
\varphi^*:C_{V_0}^\infty\lvert_{U_2}(\cdot) \to C_{M_0}^\infty\lvert_{U_1}(\cdot)
\]
is an isomorphism. Pflaum \cite[Section~6]{pflaum2001smooth} then showed that the composition of $\varphi$ with the reduced Hilbert map
\[
\overline{\sigma_V}\circ\varphi:U_1\to\R^k
\]
is a singular chart and that the collection of singular charts obtained this way forms a singular atlas on $M_0$. Moreover, the induced smooth structure has $C_{M_0}^\infty(\cdot)$ as the sheaf of smooth functions.
\end{example}

\subsection{Commutative control data for reduced spaces}\label{reduced_space_control_data_subsection}
\begin{definition}\label{fibered_neighbourhood}
Let $(A, \S)$ be a stratified space with a smooth structure and let $X\in S$. A \textbf{fibered neighbourhood} of $X$ in $A$ consists of the following data:
\begin{enumerate}
    \item A neighbourhood $T_X$ of $X$ in $A$.
    \item A \textbf{homothety} of $T_X$ to $X$ --- a smooth map 
    \begin{align*}
    m_X&:\R_{\geq 0}\times T_X\to T_X \\
    m_X^t&(v) \coloneqq m(t,v)
    \end{align*}
    with the properties:
    \begin{enumerate}
        \item $m_X^t\circ m_X^s = m_X^{ts}$.
        \item $m_X^1 = \operatorname{Id}\lvert_{T_X}$.
        \item For $t>0$, the maps $m_X^t$ preserve strata.
        \item $m_X^0(T_X)=X$.
    \end{enumerate}
    \item A \textbf{distance function} --- A smooth function $\rho:T_X\to \R_{\geq 0}$ satisfying:
    \begin{enumerate}
        \item $\rho_X\circ m_X^t = t^2\cdot \rho$.
        \item $\rho_X^{-1}(0)=X$.
    \end{enumerate}
\end{enumerate}
\end{definition}

With the notation above, assume $X<Y$ implies $\dim{X}<\dim{Y}$.  This is the case for example when $(A, \S)$ with its smooth structure is Whitney (B) regular. Let
\[
A_d \coloneqq \bigsqcup_{X\in \S_d}X
\]
and
\begin{align*}
A_{<d} &\coloneqq \bigsqcup_{i<d}A_i \\
A_{\geq d} &\coloneqq \bigsqcup_{i\geq d}A_i = A \setminus A_{<d}.
\end{align*}
and note that the extra assumption implies that $A_{<d}$ is closed in $A$.

Let ${\{\left(T_X,m_X^t,\rho_X\right)\}}_{X\in\S}$ be a collection, indexed by $X\in\S$, such that $\left(T_X,m_X^t,\rho_X\right)$ is a fibered neighbourhood of $X$. Going over the definitions of pre-commutativity, commutativity and being adjusted in Subsection \ref{control_data_subsection} and replacing $M_d$ with the open set $A_{\geq d}$, one can make sense of what it means for the collection to satisfy each of these properties. For example, commutativity now reads that for all $X,Y\in \S$:
\begin{enumerate}
    \item The homotheties $m_X^s, m_Y^t$ commute in the domain of definition of their compositions.
    \item $\rho_X\circ m_Y^t=\rho_X$ in the domain of definition.
\end{enumerate}

Note that the notion of being tangential is already included in the definition of a fibered neighbourhood --- this is exactly the requirement that the homothety $m_X^t$ preserves strata for $t>0$.
\begin{thm}\label{reduced_space_control_data_thm}
Let $M_0$ be a reduced space with its stratification and smooth structure from Example \ref{reduced_manifold_example}. Then:
\begin{enumerate}
    \item There exists a collection ${\{\left(T_X,m_X^t,\rho_X\right)\}}_{X\in\S}$ of fibered neighbourhoods that is adjusted and commutative.
    \item For any stratified subspace $C_0\subset M_0$ there exists a neighbourhood of $C_0$ in $M_0$ admitting a smooth weak deformation retraction to $C_0$. Moreover, this smooth weak deformation retraction is the restriction of a smooth homotopy $\overline{F}:[0,1]\times M_0 \to M_0$.
\end{enumerate}
\end{thm}
\begin{proof}
The $G$-equivariance of the structures obtained from Theorems \ref{zero_level_set_stratification_thm}, \ref{tangential_control_data_thm}, \ref{commutative_control_data_theorem} and \ref{weak_deformation_theorem} and Remark \ref{stratified_subspace_deformation} allows one to reduce these structures to $M_0$ and obtain the required reduced structures.
\end{proof}

\subsection{Smooth local triviality with quasi-homogeneous fibers}
\begin{definition}
Let $(A, \S)$ be a stratified space with a smooth structure. It is smoothly locally trivial with conical fibers if the images of singular charts with their induced stratifications are smoothly locally trivial with conical fibers in the sense of Definition \ref{locally_trivial}.
\end{definition}

Given this definition, one may ask whether reduced spaces with their stratifications and smooth structures from Example \ref{reduced_manifold_example} have this property. Examining the proof of \cite[Theorem~6.7]{sjamaar1991stratified} and \cite[Section~6]{pflaum2001smooth}, one can see that they are smoothly locally trivial with fibers homeomorphic to cones. But they are not smoothly locally trivial with conical fibers, as exemplified by the following:
\begin{example}[{\cite[Examples~2.3 and 2.11]{chossat_examples}}]\label{d6_representation}
Take $H=D_3$, the symmetry group of an equilateral triangle, with the natural action on $V=\R^2$. Here the momentum map is zero (so $\mu^{-1}(0)=V$) and the homogeneous generators for the algebra of invariant polynomials are
\[
\sigma_1 = x^2 + y^2,\:\sigma_2 = x^3 -3xy^2.
\]

The image $\sigma(V)$ is the subset of $\R^2$ defined by
\[
x\geq0, \:y^2\leq x^3
\]
and the images of strata are
\begin{align*}
\sigma\left(V_{(D_3)} \right) &= \{0\} \\
\sigma\left(V_{(\varepsilon)} \right) &= \{x\geq 0,\:y^2=x^3\} \\
\sigma\left(V_{(1)} \right) &= \{x> 0,\:y^2<x^3\}.
\end{align*}
where $\varepsilon$ denotes a reflection.

Around the origin, higher strata cannot be mapped diffeomorphically to a cone as the two connected components of $\sigma\left(V_{(\varepsilon)} \right)$ approach the origin from the same direction. But they do have a quasi-homogeneous structure --- the map
\[
(x,y) \mapsto (t^2\cdot x,t^3\cdot y)
\]
preserves the strata for $t>0$.
\end{example}

This motivates the following definition:
\begin{definition}\label{quasi_homogeneous_locally_trivial}
Let $(A, \S)$ be a stratified space with a smooth structure. It is \textbf{smoothly locally trivial with quasi-homogeneous fibers} if it satisfies the following property.

For each $X\in \S$ of dimension $k$, there exists a non-negative integer $l$ and a list of integers $\bar{d}_X\coloneqq \left(d_1,\ldots,d_l\right)$  called \textbf{weights} such that the following holds. For each $p\in X$ there exists a neighbourhood $U_p$ of $p$ in $C$ and a singular chart 
\[
\theta_p:U_p\to\Omega_p\subset \R^k \times \R^l
\]
such that:
\begin{enumerate}
    \item $\theta_p(X\cap U_p)=\left(\R^k\times\{0\}\right)\cap \Omega_p$.
    \item For all $Y\in\S$ with $X<Y$ we have
    \[
    \theta_p(Y\cap U_p)=\left(\R^k\times C_{Y,p}^{\bar{d}_X}\right)\cap \Omega_p
    \]
    where $C_{Y,p}^{\bar{d}_X}\subset \R^l\setminus \{0\}$ is \textbf{$\bar{d}_X$-conical}, i.e. invariant under the map
    \[
    (y_1,\ldots,y_l) \mapsto (t^{d_1}\cdot y_1,\ldots,t^{d_l} \cdot y_l)
    \]
    for $t\in (0,1)$.
\end{enumerate}
\end{definition}
\begin{definition}
A \textbf{quasi-homogeneous neighbourhood} of $p$ is the data $\left(U_p, \Omega_p, \theta_p\right)$.
\end{definition}
\begin{rem}
One can check that the local condition of Definition \ref{quasi_homogeneous_locally_trivial} implies Whitney (B) regularity. 
\end{rem}
\begin{thm}\label{reduced_space_local_structure_thm}
Let $M_0$ be a reduced space with its stratification and smooth structure from Example \ref{reduced_manifold_example}. 

It is smoothly locally trivial with quasi-homogeneous fibers.
\end{thm}
\begin{proof}
The essence and most details of the proof can be found in \cite[subsection~4.3]{duistermaat2007dynamical}, albeit for $M/G$ instead of $M_0=\mu^{-1}(0)/G$. We recall them here as they give a nice explicit description of the weights for a given stratum.

Let $Z\coloneqq \mu^{-1}(0)$ with stratification $\S$, let $X\in\S,p\in X,H=G_p$ and $V$ be the symplectic slice of $p$.
By the local nature of this property, the local normal form and the fact that the normal orbit type is constant along a stratum we can reduce to the linear case. Therefore we may assume $M=V,G=H,X=V^H,p=0$ and $Z$ is the zero set of a quadratic momentum map. Let $k$ be the dimension of $V^H$.

Denote by $\pi:V\to V^H$ the projection given by averaging over $H$ and note that the symplectic perpendicular $W$ of $V^H$ is given by $\ker\pi$. Identify ${(V^H)}^*$ as a subspace of $V^*$ using $\pi^*:{(V^H)}^*\hookrightarrow V^*$ and let $\left(\sigma_1,\ldots,\sigma_k\right)$ be a basis of ${(V^H)}^*\subset V^*$. Dimensions considerations imply that the space of linear $H$-invariant polynomials ${P(V)}^H_{(1)}$ is exactly ${(V^H)}^*$.

Now complete $\left(\sigma_1,\ldots,\sigma_k\right)$ to some homogeneous Hilbert basis $\left(\sigma_1,\ldots,\sigma_k,\widetilde{\sigma_{k+1}},\ldots,\widetilde{\sigma_{k+l}}\right)$ and let
\[
\sigma_{k+j}\coloneqq\widetilde{\sigma_{k+j}}\circ(\operatorname{Id}-\pi)
\]
for $1\leq j\leq l$. Using induction on the degree, one can show that
\[
\sigma \coloneqq \left(\sigma_1,\ldots,\sigma_k,\sigma_{k+1},\ldots,\sigma_{k+l}\right)
\]
is also a homogeneous Hilbert basis. Moreover, for $v\in V^H,w\in W$ we have
\begin{align*}
\forall 1\leq j \leq k:\: &\sigma_j(v+w) = \sigma_j(v) \\
\forall 1\leq j \leq l:\: &\sigma_{k+j}(v+w) = \sigma_j(w)
\end{align*}
so that $\sigma:V^H\times W\to \R^k\times \R^l$ is of the form
\[
\sigma(v, w) = \left(\sigma_1(v),\ldots,\sigma_k(v),\sigma_{k+1}(w),\ldots,\sigma_{k+l}(w) \right).
\]

Write $\bar{d}\coloneqq (\deg{\sigma_{k+1}},\ldots,\deg{\sigma_{k+l}})$. The map 
 $\sigma:V\to\R^k\times \R^l$ satisfies the following properties:
\begin{enumerate}
    \item $\sigma(V^H)=\R^k\times \{0\}$
    \item If $C_W\subset W\setminus \{0\}$ is a conical subset, then
    \[
    \sigma(V^H\times C_W)=\R^k\times C_W^{\overline{d}}
    \]
    for some $\bar{d}$-conical subset $C_W^{\bar{d}}\subset \R^l\setminus \{0\}$.
\end{enumerate}

Examining the proof of Theorem \ref{crit_set_stratification_thm} one sees that a stratum $Y$ of $Z\subset V$ is of the form $V^H\times C_Y$ for some conical $C_Y\subset W$. This together with the above properties imply that the reduced map $\overline{\sigma}$ induces a quasi-homogeneous neighbourhood of $p$.
\end{proof}
\begin{rem}
Let $X$ be a stratm of $Z=\mu^{-1}(0)$, let $X_0=X/G$ and let
\[
{C_M^\infty(\cdot)}_{X,(i)}^G \subset {C_M^\infty(\cdot)}^G
\]
be the ideal sheaf of $G$-invariant functions vanishing on $X$ to order $i$. Denote by 
\[
C_{M_0}^\infty(\cdot)_{X_0,(i)}\subset C_{M_0}^\infty(\cdot)
\] 
the ideal sheaf that is the image of ${C_M^\infty(\cdot)}_{X,(i)}^G$ under the surjective map 
\[
{C_M^\infty(\cdot)}^G \to C_{M_0}^\infty(\cdot).
\]

Using the notation and the local structure from the proof of Theorem \ref{reduced_space_local_structure_thm}, one can show that in a neighbourhood $U_0$ of $x\in X_0$ the ideal $C_{M_0}^\infty(U_0)_{X_0,(i)}\subset C_{M_0}^\infty(U_0)$ is generated by the reduced monomials
\[
\prod_{j=1}^l\bar{\sigma}_{k+j}^{s_j}
\]
with 
\[
\sum_{j=1}^l \deg{\sigma_{k+j}}\cdot s_j \geq i.
\]

We get that the filtration
\[
C_{M_0}^\infty(\cdot)_{X,(1)}\supset C_{M_0}^\infty(\cdot)_{X,(2)}\supset \dots
\]
of $C_{M_0}^\infty(\cdot)$ can be seen as a \textit{weighting along $X_0$} in the sense of Loizides and Meinrenken \cite{LOIZIDES2023109072}, but in a singular setting. Moreover, for a homothety $m_{X_0}^t$ from Theorem \ref{reduced_space_control_data_thm} and a singular chart $(\theta, U, \Omega)$ we believe that the map
\[
\theta \circ m_{X_0}^t \circ \theta^{-1}: \theta(U)\to \theta(U)
\]
is the restriction to $\theta(U)$ of a homothety coming from a \textit{weighted tubular neighbourhood along $\theta(X_0)$ in $\Omega$}.
\end{rem}
\begin{rem}
We believe the following to be true:
\begin{enumerate}
    \item Orbispaces of proper Lie groupoids, with the stratification and smooth structure from \cite{orbi_spaces_groupoids}, are smoothly locally trivial with quasi-homogeneous fibers. In particular so are orbifolds.
    \item Theorem \ref{reduced_space_control_data_thm} can be proved for any stratified space with a smooth structure that is smoothly locally trivial with quasi-homogeneous fibers, by altering the definitions and ideas from Sections \ref{tangential_control_data_section}, \ref{commutative_control_data_section} and \ref{weak_deformation_section} to fit the non-embedded, quasi-homogeneous setting.
\end{enumerate}
Details will appear elsewhere.
\end{rem}
\printbibliography
\end{document}